\documentclass[12pt]{amsart}
\usepackage{graphicx}
\usepackage[english]{babel}

\usepackage[utf8]{inputenc}
\usepackage{amsmath,amsfonts,amssymb,amscd,amsthm,txfonts,hyperref}
\usepackage[all]{xy}
\usepackage{multibib}
\usepackage{cleveref}
\usepackage[shortlabels]{enumitem}
\usepackage{comment}

\newtheorem{theorem}{Theorem}[section]
\newtheorem{proposition}{Proposition}[section]
\newtheorem{lemma}{Lemma}[section]
\newtheorem{corollary}{Corollary}[section]
\newtheorem{remark}{Remark}[section]

\theoremstyle{definition}

\newcommand{\C}{\varmathbb C}
\newcommand{\R}{\varmathbb R}

\newcommand{\Z}{\varmathbb Z}
\newcommand{\Q}{\varmathbb Q}

\newcommand{\N}{\varmathbb N}

\newcommand{\E}{\varmathbb E}

\newcommand{\parmi}[2]{\begin{pmatrix} #2 \\ #1 \end{pmatrix}}
\newcommand{\lap}{\bigtriangleup}

\DeclareMathOperator{\supp}{supp}

\renewcommand{\det}{\mbox{det}}

\newcommand{\an}[1]{\langle #1 \rangle}
\newcommand{\la}{\langle}
\newcommand{\ra}{\rangle}

\title[Stability of equilibria for Hartree-Fock]{Stability of homogeneous equilibria\\ of the Hartree-Fock equation,\\ for its equivalent formulation for random fields}
\date{}

\author{Charles Collot}
\address{Laboratoire Analyse, G\'eom\'etrie et Mod\'elisation,\\
CY Cergy Paris Universit\'e,\\
2 avenue Adolphe Chauvin, 95300, Pontoise, France}
\email{ccollot@cyu.fr}
\author{Elena Danesi}
\address{Universit\'a degli studi di Padova,\\
Dipartimento di Matematica,\\
Via Trieste, 63, 35131 Padova PD, Italy}
\email{edanesi@math.unipd.it}
\author{Anne-Sophie de Suzzoni}
\address{Centre Mathématique Laurent Schwartz,\\
\'Ecole Polytechnique, CNRS, Universit´e Paris-Saclay,\\
Palaiseau, 91128 Cedex, France}
\email{anne-sophie.de-suzzoni@polytechnique.edu}
\author{Cyril Malézé}
\address{Centre Mathématique Laurent Schwartz,\\
\'Ecole Polytechnique, CNRS, Universit´e Paris-Saclay,\\
Palaiseau, 91128 Cedex, France}
\email{cyril.maleze@polytechnique.edu}

\begin{document}

\maketitle

\begin{abstract}

The Hartree-Fock equation admits homogeneous states that model infinitely many particles at equilibrium. We prove their asymptotic stability in large dimensions, under assumptions on the linearised operator. Perturbations are moreover showed to scatter to linear waves. We obtain this result for the equivalent formulation of the Hartree-Fock equation in the framework of random fields. The main novelty is to study the full Hartree-Fock equation, including for the first time the exchange term in the study of these stationary solutions.
    
\end{abstract}

\section{Introduction}

\subsection{The Hartree-Fock equation in the framework of random fields} \label{subsec:time-dependent-hartree-fock}

We consider the time-dependent Hartree-Fock equation, which models large fermionic systems. We show it admits equilibria that are space and time translations invariant, which corresponds to an homogeneous electron gas. We prove their asymptotic stability under localised perturbations. We formulate the time-dependent Hartree-Fock equation in the setting of random fields (see \cite{dS} and the formal derivation in Subsection \ref{subsec:formal-derivation}):
\begin{equation}\label{Cauchyprob}
\left\{ \begin{array}{l l} & i\partial_t X = -\Delta X + (w*\E [|X|^2]) X - \int_{\R^d}  w(x-y)\E[ \overline{ X (y)} X(x)] X(y) \, dy ,\\
& X(t=0)=X_0
\end{array} \right.
\end{equation}
where $X:[0,T]\times \mathbb R^d \times \Omega\mapsto \mathbb C$ is a random field defined over a probability space $(\Omega, \mathcal A,\mathbb P )$, $w$ is an even pairwise interaction potential, $\E$ denotes the expectancy on $\Omega$. Equation \eqref{Cauchyprob} is an equivalent reformulation of the standard Hartree-Fock equation for density matrices
\begin{equation} \label{Cauchyprob-density-matrices}
i\partial_t \gamma=[-\Delta+w*\rho+\mathcal X, \gamma]
\end{equation}
for a nonnegative self-adjoint operator $\gamma$ on $L^2(\mathbb R^d)$ with kernel $k$, density $\rho(x)=k(x,x)$ and where $\mathcal X$ denotes the exchange term operator with kernel $-w(x-y)k(x,y)$. Indeed, for a given solution $X$ to \eqref{Cauchyprob}, if $\gamma$ denotes the operator whose kernel $k$ is the correlation function $k(x,y)=\mathbb E (X(x)\overline{X(y)})$, then $\gamma$ is a solution to \eqref{Cauchyprob-density-matrices}. And conversely, for a non-negative finite rank self-adjoint operator $\gamma =\sum_{i=1}^N |\varphi_i \rangle \langle \varphi_i|$ that is a solution of \eqref{Cauchyprob-density-matrices}, the associated random field $X=\sum_{i=1}^N g_i \varphi_i$, where $g_1,...,g_N$ are centered normalised and independent gaussian variables, solves \eqref{Cauchyprob} (this generalises to infinite rank self-adjoint operators using the spectral Theorem). The Cauchy problem for solutions to \eqref{Cauchyprob-density-matrices} emanating from localised initial data was solved in \cite{BDF,BDF2,chadam1976time,CG,Z}.

Our motivation to study the formulation \eqref{Cauchyprob} of the Hartree-Fock equation instead of \eqref{Cauchyprob-density-matrices} is to draw analogies with the nonlinear Schr\"odinger equation. The randomness in \eqref{Cauchyprob} lies in the use of the extra variable $\omega \in \Omega$ that we think of as a probability variable, which is a convenient way to represent the coupling for this system of equations. It is also convenient for obtaining representation formulas for the equilibria, see Subsection \ref{subsec:equilibria}. We remark that equation \eqref{Cauchyprob} is not a deterministic problem that we solve for a random data indexed by $\omega\in\Omega$. Instead, due to the presence of the expectancy $\E$, \eqref{Cauchyprob} is a system of (potentially infinite many) equations coupled by the variable $\omega$.  

The time-dependent Hartree-Fock equation is a mean-field equation for the dynamics of large Fermi systems. We recall in Subsection \ref{subsec:formal-derivation} a quick formal mean-field derivation of \eqref{Cauchyprob} from the many-body Schr\"odinger equation. The rigorous derivation of \eqref{Cauchyprob-density-matrices} in the mean-field limit was first done in \cite{BardosDerivation}, and was extended to the case of unbounded interaction potentials, as the Coulomb one, in \cite{FrohlichDerivationFermi}. Estimates for the convergence in the semi-classical limit that arises for confined Fermi systems were proved in \cite{EESY,benedikter2014mean}, and were extended recently to mixed states, and conditionnaly to more singular interaction potentials, in \cite{Benedikter2016Dec,porta2017mean}. Another derivation by different techniques, and including another large volume regime for long-range potentials, was given in \cite{Petrat2016Mar,BBPPT}.

The time-dependent Hartree-Fock equation \eqref{Cauchyprob} and \eqref{Cauchyprob-density-matrices} has received less attention than the static Hartree-Fock equation (obtained by setting $\partial_t X=0$ or $\partial_t \gamma=0$ in \eqref{Cauchyprob} or \eqref{Cauchyprob-density-matrices} respectively), or the reduced Hartree-Fock equation \eqref{Cauchyprob-reduced} and \eqref{Cauchyprob-density-matrices-reduced} (obtained by suppressing the last term in \eqref{Cauchyprob} or \eqref{Cauchyprob-density-matrices}, which is the exchange term). We mention that the exchange term in the Hartree equation does not often appear in the case of bosonic systems; it is absent in Bose--Einstein condensate. This is due to the form of the canonical wave function considered for bosonic systems. The form of the canonical wave function differs for fermionic systems, and the exchange term always appears. But this exchange term is often negligible compared to the direct term (see \cite{cances2014mathematical,Petrat2016Mar} for example), however in some cases it is relevant to keep it and study it \cite{Saffirio}.  Still the time-dependent Hartree-Fock equation is always a better approximation of the dynamics of the fermions than the reduced Hartree equation. 

\subsection{Stability of non-localised equilibria for the reduced Hartree-Fock equation}
Most of the results describing the dynamics have so far been obtained for the equations without exchange term 
\begin{equation}\label{Cauchyprob-reduced}
i\partial_t X = -\Delta X + (w*\E [|X|^2]) X ,
\end{equation}
and
 \begin{equation} \label{Cauchyprob-density-matrices-reduced}
i\partial_t \gamma=[-\Delta+w*\rho, \gamma],
\end{equation}
which are sometimes referred to as the reduced Hartree-Fock equation as in \cite{CDL}. Indeed, the contribution of the exchange term is negligible in certain regimes. It is indeed negligible in the semi-classical limit to the Vlasov equation, which was studied in \cite{GIMS98} for a system with a finite number of particles, and in \cite{Benedikter_2016}, \cite{benedikter2014mean} in the limit of number of particles going to infinity.


To obtain a better accuracy, instead of suppressing the exchange term, in Density Functional Theory one approximates it as a function of the density $\rho=\E[\lvert X\rvert^2]$. This leads to the Kohn-Sham equations. We refer to \cite{CF} for a review and \cite{J,SCB} for the Cauchy problem of time-dependent Kohn-Sham equations. The stability of the zero solution for the time-dependent Kohn-Sham equation was showed in \cite{PS}. 

The reduced Hartree equations \eqref{Cauchyprob-reduced} and \eqref{Cauchyprob-density-matrices-reduced} admit nonlocalised equilibria that models a space-homogeneous electron gas. The stability of such equilibria has been studied in \cite{CHP,CHP2,LS2,LS1} for Equation \eqref{Cauchyprob-density-matrices-reduced} and \cite{CdS,CdS2,H} for Equation \eqref{Cauchyprob-reduced}, under suitable assumptions for the potential. In \cite{NAKASTUD}, Hadama proved the stability of steady states for the reduced Hartree equation in a wide class, which includes Fermi gas at zero temperature in dimension greater than $3$, with smallness asumption on the potential function. In all these works, the linearised equation for the density around equilibrium (often called the linear response) is an invertible space-time Fourier multiplier. In a recent work \cite{NY23}, Nguyen and You proved that this Fourier multiplier could not be inverted in the case of the Coulomb interaction potential. However, they were still able to describe and to prove some time decay via dispersive estimates for the linearised dynamics around steady states. 

As observed in \cite{LS2020}, there is a natural way to associate nonlocalised equilibria of \eqref{Cauchyprob-density-matrices-reduced} to space-homogeneous equilibria of the Vlasov equation. In the same article, the authors proved that, in the high density limit, the Wigner transform (see \cite{LT93} for a survey on the Wigner transform) of solutions of \eqref{Cauchyprob-density-matrices-reduced} close to equilibria converge towards solutions of the Vlasov equation. Moreover, these solutions remain close, in a suitable sense, to the corresponding classical equilibria of the Vlasov equation (see \cite{LS2020}, Thm 2.22).
As mentioned before in this introduction, the exchange term is negligible in the semi-classical limit; then we expect that a similar result could be proven considering the Hartree equation with the exchange term. This will be the object of future works.

\subsection{Main result}

In this article we study, to our knowledge for the first time, non-linear asymptotic dynamics of the Hartree-Fock equation, including the exchange term. Our first result is that Equation \eqref{Cauchyprob} admits non-localised equilibria of the form
\begin{equation}\label{equilibria}
    Y_f=\int_{\mathbb R^d} f(\xi)e^{i(\xi x-\theta(\xi)t)}dW(\xi).
\end{equation}
Above, the momenta distribution function $f$ is any nonnegative $L^2$ function, the phase is 
\begin{equation} \label{id:def-theta}
    \theta(\xi)= \lvert \xi\rvert^2+ \int_{\mathbb R^d} w dx \, \int_{\mathbb R^d} f^2dx- (2\pi)^{d/2} (\hat w *f^2)(\xi),
\end{equation}
where $\hat w(\xi)=(2\pi)^{-d/2}\int e^{-i\xi x}w(x)dx$ is the Fourier transform of $w$, and $dW$ is the Wiener integral (see \cite{Janson} for more information on its construction) i.e. $dW(\xi)$ are infinitesimal centred Gaussian variables characterized by
\[
\mathbb E [\overline{dW(\xi)} dW(\xi')]=\delta(\xi-\xi')d\xi d\xi'.
\]
The probability law of the Gaussian field $Y_f$ is invariant under time and space translations. It is an equilibrium of \eqref{eqprinc} in the sense that it is a solution whose probability law does not depend on time. Complete details on these equilibria are given in Subsection \ref{subsec:equilibria}. We consider perturbations of these equilibria with initial data
\begin{equation} \label{id:initial-data}
X_{|t=0} = Y_f + Z_0 .
\end{equation}
We show asymptotic stability of the equilibrium via scattering to linear waves given by the evolution
$$
S(t)\, Z(x) = \frac{1}{(2\pi)^{d/2}} \int_{\mathbb R^d} e^{i(\eta x-\theta(\eta)t)}\hat Z(\eta)\, d\eta .
$$
The linear flow $S(t)$ disperses as does the Schrödinger linear flow, both pointwise and regarding Strichartz estimates. Details are given in Section \ref{sec:linear}. We set $\tilde \theta = -(2\pi)^{d/2} \hat w * f^2$ and assume
\begin{equation} \label{id:hp-regularite-tilde-theta}
\tilde \theta \in C^{d+2}\cap W^{d+2,\infty}\cap W^{4,1},
\end{equation}
and the uniform ellipticity assumption
\begin{equation} \label{id:uniform-ellipticity-theta}
 \eta^\top \, \nabla^{\otimes 2}\theta (\xi)\, \eta >\lambda_*|\eta|^2
\end{equation}
for all $\xi,\eta\in \mathbb R^d$, $\eta \ne 0$, for a constant $\lambda_*>0$. Above in \eqref{id:uniform-ellipticity-theta}, $\nabla^{\otimes 2}$ denotes the Hessian matrix and $a^\top b=\sum_{k=1}^d a_k b_k$. We denote by $\| \mu \|_M$ the total variation of a measure $\mu$.

\begin{theorem}\label{th:main} Let $d\geq 4$. Let $w$ be an even Borel measure with $\an{y} w$ a finite measure, and $g=f^2$ that satisfies $g \in W^{3,1}\cap W^{3,\infty}$ and $\langle \xi\rangle^{2\lceil s_c \rceil}g\in W^{2,1} $, be such that $\theta$ given by \eqref{id:def-theta} satisfies \eqref{id:hp-regularite-tilde-theta} and \eqref{id:uniform-ellipticity-theta}.

Then there exists $C(\| \nabla^{\otimes 3}\Tilde \theta \|_{L^{\infty}},\lambda_*)>0$, that is increasing in $\| \tilde \theta\|_{\dot W^{3,\infty}}$ and decreasing with $\lambda_*$, such that assuming
\begin{equation}\label{smallness-assumption}
\| \la y \ra w \|_{M} \| \nabla g \|_{W^{2,1}}\leq C(\| \nabla^{\otimes 3}\Tilde \theta \|_{L^{\infty}},\lambda_*),
\end{equation}
the following holds true. There exists $\delta>0$ such that if
\[
 \|Z_0\|_{L^2(\Omega,H^{s_c}(\R^d))\cap L^{2d/(d+2)}(\R^d,L^2(\Omega))}\leq \delta,
\]
 the Cauchy problem \eqref{Cauchyprob} with initial data \eqref{id:initial-data} has a solution $X=Y_f+Z$ with
\[
Z \in \mathcal C(\R, L^2(\Omega,H^{s_c}(\R^d)))
\]
and there exist $Z_{\pm} \in L^2(\Omega, H^{s_c}(\R^d))$ such that 
\[
X(t) = Y_f(t) + S(t) Z_\pm + o_{L^2(\Omega,H^{s_c}(\R^d))}(1)
\]
when $t\rightarrow \pm \infty$.
\end{theorem}

\begin{remark}

The solution of Theorem \ref{th:main} is unique in a certain class of global in time perturbations: that which are small in the space defined by \eqref{normZ} and whose density is small in the space given by \eqref{normV}. This is  as a by product of our proof by Banach fixed point. We believe uniqueness holds in $\mathcal C(\R, H^s) \cap L^{2\frac{d+2}{d}}(\R, W^{s_c,2\frac{d+2}{d}}(\R^d))$ without smallness assumption by a local well-posedness argument similar to that for the cubic NLS, see for instance \cite{GinibreVelo}.

\end{remark}

\noindent \textbf{Comments on the result} 

\smallskip

\noindent \emph{On the optimality of the regularity assumptions.} We remark that $s_c$ is the critical Sobolev regularity for the cubic Schrödinger equation in dimension $d$. Thus, our regularity assumption on $Z_0$ seems optimal.

\medskip

\noindent \emph{On the perturbative nature of the result, compared with the reduced Hartree-Fock equations.} This result is perturbative since the equilibrium has to satisfy the smallness assumption \eqref{smallness-assumption}, which contrasts with the previous results of the first and third author, \cite{CdS,CdS2}. The reason for this is twofold: because of the exchange term, first, the dispersion relation $\theta$ can loose its ellipticity, and second, the linearised equation for the density $\rho$ around the equilibrium is no longer a Fourier multiplier, see Section \ref{sec:linear}. We rely on perturbative arguments to invert it. However, Theorem \ref{th:main} covers "large" equilibria, in the sense that $\| g\|_{L^1}$ can be large so that the density is large.

\medskip

\noindent \emph{Some examples covered by the result.} We now describe a physically relevant example of such $g$ for which Theorem \ref{th:main} applies. For Fermi gases at density $\rho$ and positive temperature, the function $g$ depends on the temperature $T$ and the chemical potential $\mu$ as
\[
g[\rho,T,\mu] (\xi) = \rho C_{\frac{\mu}{T}} T^{-\frac d2} \frac1{e^{(|\xi|^2-\mu)/T}+1}
\]
where $C_{\tilde \mu}=(\int \frac{d\xi}{e^{|\xi|^2-\tilde \mu}+ 1})^{-1}$. We can figure several ways of satisfying the assumptions \eqref{id:uniform-ellipticity-theta} and \eqref{smallness-assumption}. Indeed one estimates that for $s=1,2,3$
\[
\| \nabla^s g \|_{L^1} =\rho T^{-\frac s2} \frac{\|  \nabla^s \frac{1}{e^{|\xi|^2-\mu/T}+ 1}\|_{L^1}}{\| \frac{1}{e^{|\xi|^2-\mu/T}+ 1}\|_{L^1}}\approx \rho T^{-\frac s2} \left(\frac{\mu}{T} \right)_+^{\frac s2 -1}
\]
and for $s=2,3$
\begin{align*}
\| \nabla^s \tilde \theta \|_{L^\infty} & \leq \| w\|_{M}\| \nabla^s g\|_{L^1}\lesssim \rho T^{-\frac s2} \left(\frac{\mu}{T} \right)_+^{\frac s2 -1}.
\end{align*}
The hypotheses of the Theorem are then met for a small enough density at fixed temperature and chemical potential, or at fixed density letting $T$ go to infinity while maintaining $\frac\mu{T}$ negative or constant positive.

\medskip

\noindent \textbf{Comments on the proof}

\smallskip

As already mentionned, the result is perturbative and hence so is the proof. The articulation is the following, in Section \ref{sec:settingContraction} we write the equation as a system of two equations: one equation describes the full pertubation as a linear equation depending on its correlations, the other equation describes the evolution of the correlations. Contrary to \cite{CdS,CdS2}, one needs to describe the full set of correlations, that is $\E(\bar X(x) X(y))$ and not only the diagonal $\E(|X(x)|^2)$, see Section \ref{sec:settingContraction}. 

The treatment of the nonlinear terms is similar to \cite{CdS}. The novelty is that we have two space variables instead of one. 

What mostly differs is the treatment of the linearised around the equilibrium part. It requires proving Strichartz estimates for a propagator that is adapted to the exchange term, see Subsection \ref{subsec:Strichartz}. The rest of the treatment of the linearised equation is based on explicit computations and functional analysis see Subsection \ref{subsec:formulas} and Proposition \ref{pr:continuity-L4}.

\medskip

\noindent \textbf{Organisation of the paper}

\smallskip

In the rest of the introductive section, we formally derive the Hartree equation with exchange term and we explain what are its equilibria. In Section \ref{sec:notations}, we give useful notations and conventions used throughout the paper. In Section \ref{sec:settingContraction}, we rewrite the problem as a contraction argument and we specify the functional framework. In Section \ref{sec:linear}, we treat the linearised equation around its equilibria. In Section \ref{sec:bilinear}, we prove bilinear estimates that are sufficient to close the contraction argument. Finally, in Section \ref{sec:fixed-point}, we perform the contraction argument and give the final arguments to prove Theorem \ref{th:main}.

\subsection{Formal derivation of the equation from the \texorpdfstring{$N$}{N} body problem}

\label{subsec:formal-derivation}

In this subsection, we present a quick and formal derivation of Equation \eqref{Cauchyprob} from the many-body Schr\"odinger equation. The references for rigorous results are given in Subsection \ref{subsec:time-dependent-hartree-fock}. We consider $N$ particles, represented by a wave function $\psi : \R^{dN} \rightarrow \C$, with binary interactions through a pair potential $w$. The Hamiltonian of the system is
\begin{eqnarray*}
     \mathcal{E}_N(\psi) & = & -\int_{\R^{dN}}\Bar{\psi}\sum_{i=1}^N \Delta_i\psi \, d\underline x +\int_{\R^{dN}}\sum_{i<j}w(x_i-x_j)\lvert \psi\rvert^2 \, d\underline{x}\\
       & = & \mathcal{E}_{kin}+\mathcal{E}_{pot}
\end{eqnarray*}

where $\underline x=(x_1,...,x_N)\in \mathbb R^{dN}$, each $x_i$ belongs to $\R^d$ and $\Delta_i$ is the Laplacian in the variable $x_i$. We consider for fermions wave functions of the form of a Slater determinant
\[
\psi(\underline x)=\frac{1}{\sqrt{N!}}\sum_{\sigma\in \mathfrak{S}_N}\varepsilon(\sigma)^\iota\prod_{j=1}^N u_{\sigma(j)}(x_j),
\]
where $(u_i)_{1\leq i\leq N}$ is an orthonormal family in $L^2(\R^d)$, and the sum is over permutations $\sigma$ of $\{1,N\}$ whose signature is denoted by $\varepsilon (\sigma)$. We recall that these standard Ansätze are driven by the ideas that bosons are indiscernible particles and that fermions satisfy Pauli's exclusion principle, which translates to the symmetry and skew symmetry of $\psi$ respectively. Under these Ans\"atze, the expression of the Hamiltonian simplifies. Indeed, the potential energy becomes
\begin{eqnarray*}
\mathcal{E}_{pot}(\psi)&=&\int_{\R^{dN}}\frac{1}{N!}\sum_{i<j}w(x_i-x_j)\sum_{\sigma_1,\sigma_2\in\mathfrak{S}_N}\varepsilon(\sigma_1\sigma_2)\prod_{k=1}^N\overline{u_{\sigma_1(k)}(x_k)}u_{\sigma_2(k)}(x_k)\, d \underline x \\
&= & \frac{1}{N!}\sum_{\sigma_1,\sigma_2\in\mathfrak{S}_N}\varepsilon(\sigma_1\sigma_2) \sum_{i<j}\an{u_{\sigma_1(i)}\bar u_{\sigma_2(i)}, w* (\bar u_{\sigma_1(j)} u_{\sigma_2(j)})}\prod_{k\neq i,j} \an{u_{\sigma_1(k)}, u_{\sigma_2(k)}}
\end{eqnarray*}
where the scalar product is taken into $L^2(\R^d)$. Because the family $(u_k)_{1\leq k \leq N}$ is orthonormal we get that
and
\[
\prod_{k\neq i,j} \an{u_{\sigma_1(k)}, u_{\sigma_2(k)}} = \left \lbrace{\begin{array}{l l l}
1     & \quad \textrm{if } \sigma_1 = \sigma_2 \textrm{ or } \sigma_2 = \sigma_1 \circ (ij), \\
0     & \quad \textrm{otherwise} .
\end{array}}\right. 
\]
This further yields
\begin{eqnarray*}
\mathcal{E}_{pot}(\psi)&=& \sum_{i<j}\Big(\an{|u_{\sigma_1(i)}|^2, w* | u_{\sigma_1(j)}|^2} -\an{u_{\sigma_1(i)}\bar u_{\sigma_1(j)}, w* (\bar u_{\sigma_1(j)} u_{\sigma_1(i)})} \Big) .
\end{eqnarray*}
We perform the change of variables $i' = \sigma_1(i)$ and $j' = \sigma_1(j')$ and finally get
\[
\mathcal{E}_{pot}(\psi)=\frac12 \sum_{i'\neq j'}\Big(\an{|u_{i'}|^2, w*  |u_{j'}|^2 } - \an{u_{i'}\bar u_{j'}, w* (\bar u_{j'} u_{i'})} \Big) .
\]
A similar and simpler computation shows that the kinetic energy is
$$
\mathcal{E}_{kin}(\psi)=-\sum_{j'=1}^N\an{u_{j'},\lap u_{j'}} .
$$
The Hamiltonian is therefore
\[
\mathcal{E}_N(\psi)=-\sum_{j'=1}^N\an{u_{j'},\lap u_{j'}} 
 +\frac12 \sum_{i'\neq j'}\Big(\an{|u_{i'}|^2, w*  |u_{j'}|^2 } - \an{u_{i'}\bar u_{j'}, w* (\bar u_{j'} u_{i'})} \Big) .
\]

In the limit of large number of particles $N$, we formally replace the sum $\sum_{i'\neq j'}$ by $\sum_{i',j'}$ and we arrive at the following system of evolution equations for $(u_1,...,u_N)$:

\begin{equation}
i\partial_tu_j(x)=-\Delta u_j(x)+ \sum_{k}\int_{\mathbb R^d}  w(x-y)\Big(\lvert u_k(y)\rvert^2u_j(x) - \overline{u_k(y)}u_j(y) u_k(x)\Big)\, dy,
\end{equation}
for $j=1,...,N$.

One can recast the above system as an equation on random fields. Indeed, considering $X:\R\times\R^d\times\Omega\to \C$ a time dependent random field over $\R^d$ , where $(\Omega,\mathcal{A},d\omega)$ is the underlying probability space, of the form $X(x,\omega):=\sum_k u_k(x)g_k(\omega)$ where $(g_k)_{1\leq k\leq N}$ is an orthonormal family of $L^2(\Omega)$. Then $X$ is a solution of
\begin{equation}\label{eqprinc}
    i\partial_t X=-\Delta X+\Big(w*\E[\lvert X\rvert^2]\Big)X-\int_{\mathbb R^d} w(\cdot-y)\E[\overline{X(y)}X(\cdot)]X(y)dy.
\end{equation}

\subsection{Equilibria} \label{subsec:equilibria}

We briefly study in this subsection the equilibria
$$
Y_f=\int_{\mathbb R^d} f(\xi)e^{i(\xi x-\theta(\xi)t)}dW(\xi),
$$
and show they are solutions to \eqref{Cauchyprob}. Note that $f\in L^2(\mathbb R^d)$ can be chosen real and nonnegative without loss of generality. We set
\[
g=f^2
\]
and decompose the phase
\begin{equation}
    \theta(\xi)= \lvert \xi\rvert^2+\tilde \theta(\xi)+\theta_0
\end{equation}
where
\begin{equation} \label{id:def-tildetheta-theta0}
\tilde \theta = -(2\pi)^{\frac d2} \hat{w}*g \qquad \mbox{and} \qquad \theta_0=(2\pi)^d \hat{w}(0)\hat{g}(0).
\end{equation}
The function $Y_f$ is a Gaussian field, whose law is invariant under spatial translations, and its correlation function is for all $t\in \R, x,y \in \R^d$,
\begin{equation} \label{id:correlation-function-Yf}
\mathbb E [\overline{Y_f(t,x)}Y_f(t,y)]=\int_{\mathbb R^d} f^2(\xi)e^{i \xi(y- x)}d\xi = (2\pi)^{d/2}\hat{g}(x-y).
\end{equation}
This formula can be used to show that $Y_f$ is a solution to \eqref{eqprinc}. Indeed, it gives
\[
\Big(w*\E[\lvert Y_f\rvert^2]\Big)(x)= (2\pi)^d \hat w(0)\hat g(0)
\]
as well as
\begin{align*}
 \int_{\mathbb R^d} w(x-y)\E[\overline{Y_f(y)}Y_f(x)]Y(y)dy& = \int_{\mathbb R^{3d}} w(x-y) g(\xi')e^{i\xi'(x-y)} f(\xi)e^{i(\xi y-\theta(\xi)t)} d\xi' dW(\xi) dy \\
 &= (2\pi)^{\frac d2} \int \hat w (\xi-\xi') g(\xi') f(\xi)e^{i(\xi x-\theta(\xi)t)}  d\xi' dW(\xi)
 \\
 &= (2\pi)^{\frac d2} (\hat w*g)(D) Y_f
\end{align*}
and the result follows by injecting these two identities in \eqref{eqprinc}. In addition, the law of $Y_f$ is invariant under time translations, making it an equilibrium of the equation. Relevant equilibria to the present article are discussed in the comments after Theorem \ref{th:main}.

\subsection{Acknowledgments}

This result is part of the ERC starting grant project FloWAS that has received funding from the European Research Council (ERC) under the European Unions Horizon 2020 research and innovation program (Grant agreement No. 101117820). This work was supported by the BOURGEONS project, grant ANR-23-CE40-0014-01 of the French National Research Agency (ANR). Charles Collot is supported by the CY Initiative of Excellence Grant ``Investissements d'Avenir" ANR-16-IDEX-0008 via Labex MME-DII, and by the ANR grant ``Chaire Professeur Junior"  ANR-22-CPJ2-0018-01.
Anne-Sophie de Suzzoni is supported by the S. S. Chern Young Faculty Award funded by AX.
Elena Danesi is supported by the INDAM-GNAMPA Project CUP-E53C22001930001.

\section{Notation}\label{sec:notations}

\noindent The scalar product on $\mathbb R^d$ is denoted by
$$
\xi x=\sum_{1}^d \xi_i x_i.
$$
Our notation for the Fourier transform is
$$
\hat f(\xi) =\mathcal F f (\xi)=(2\pi)^{-\frac d2}\int_{\mathbb R^d} e^{-ix\xi}f(x)\, dx.
$$
Fourier multipliers by a symbol $s$ are denoted by $s(D)$ and defined by
$$
\mathcal F(s(D) f)(\xi)=s(\xi)\hat f (\xi).
$$
\\
By $\nabla^{\otimes 2}$, we denote the Hessian matrix.
\\
In order to lighten the notations we denote by $2^{\N}$ the set $\{ 2^n, n \in \N \}$.
\\
For $z \in \R^d$, we write $T_z$ the translation such that for any function $h$ and any $x\in \R^d$, $T_z h(x) = h(x+z)$.
\\
For $p,q\in [1,\infty]$ and for $s\in\R$ we denote $L_t^pW_x^{s,q}L_\omega^2$ the space $$(1-\Delta_x)^{-\frac{s}{2}}L^p(\R,L^q(\R^d,L^2(\Omega))),$$
with the norm: $$\lvert\lvert u\rvert\rvert_{L_t^pW_x^{s,q}L_\omega^2}=\lvert \lvert \langle \nabla\rangle^{s}u\rvert\rvert_{L_t^pL_x^{q}L_\omega^2}.$$
\\
In the case $q=2$ we also write $L_t^pH_x^s, L_\omega^2=L_t^pW_x^{s,2},L_\omega^2.$ 
\\
For $p,q\in[1,\infty],\ s,t\in\R$ we denote by $L^p_tB_q^{s,t}L_\omega^2=L^pB_q^{s,t}L_\omega^2$ the space induced by the norm: $$\lvert\lvert u\rvert\rvert_{L^pB_q^{s,t}L_\omega^2}=\bigg \lVert \bigg(\sum_{j<0}2^{2js}\lvert\lvert u_j\rvert\rvert^2_{L_x^qL_\omega^2}+\sum_{j\geq 0}2^{2jt}\lvert\lvert u_j\rvert\rvert^2_{L_x^qL_\omega^2}\bigg)^{\frac{1}{2}}\bigg\rVert_{L^p(\R)}.$$
\\
We denote by $M$ the space of finite signed Borel measures and endow it by the total variation $\| \cdot \|_M$.

\section{Setting the contraction argument}\label{sec:settingContraction}

In this section, we write the problem at hand, namely solving a Cauchy problem and a scattering problem, as a fixed point problem. This fixed point problem will be solved using a contraction argument in Section \ref{sec:fixed-point}, and the proof of Theorem \ref{th:main} will follow. Because $Y_f$ is not localised and thus not in any Sobolev space, neither is $X$, and we choose as a variable for the fixed point not the full solution $X$ but its perturbation around $Y_f$, namely $Z = X - Y_f$.

We fix an equilibrium $Y$, and drop the $f$ subscript to lighten the notation. We consider a perturbed solution $X=Y+Z$ to \eqref{eqprinc}. We expand using \eqref{id:correlation-function-Yf} and \eqref{id:def-tildetheta-theta0}
\begin{align*}
\mathbb E [w*|X|^2]X-\mathbb E[w*|Y|^2]Y  & =\mathbb E[w*|Y|^2] Z+ \mathbb E [w*(|X|^2-|Y|^2|)]X \\
&= \theta_0 Z+ \mathbb E [w*(|X|^2-|Y|^2|)]X
\end{align*}
and
\begin{align*}
&\int_{\mathbb R^d} w(x-y) \mathbb E [\overline{X(y)}X(x)]X(y)\, dy - \int_{\mathbb R^d} w(x-y) \mathbb E [\overline{Y(y)}Y(x)]Y(y)\, dy  \\
&\quad = \int_{\mathbb R^d} w(x-y) \mathbb E[\overline{Y(y)}Y(x)]Z(y)\, dy +\int_{\mathbb R^d} w(x-y) \mathbb E[\overline{X(y)}X(x)-\overline{Y(y)}Y(x)]X(y)\, dy \\
&\qquad = \tilde \theta(D) Z +\int_{\mathbb R^d} w(x-y) \mathbb E[\overline{X(y)}X(x)-\overline{Y(y)}Y(x)]X(y)\, dy.
\end{align*}

Hence $Z$ satisfies
\begin{eqnarray*}
    i\partial_t Z & = & \theta (D) Z + \mathbb E [w*(|X|^2-|Y|^2|)]X-\int_{\mathbb R^d} w(\cdot -y) \mathbb E[\overline{X(y)}X(\cdot )-\overline{Y(y)}Y(\cdot )]X(y) \, dy.
\end{eqnarray*}

We introduce the perturbation of the two-point correlation function
\begin{align}
\label{id:def-V} V(x,z) & =\E [\overline{X}(x+z)X(x)-\overline{Y}(x+z)Y(x)]\\
\nonumber &=\E[\overline{Y(x+z)}Z(x)+\overline{Z(x+z)}Y(x)+\overline{Z(x+z)}Z(x)].
\end{align}

The evolution equation for $Z$ becomes
\[
i\partial_t Z= \theta(D) Z+\Big(w*V(\cdot,0)\Big)(Y+Z)-\int_{\mathbb R^d}  w(z)V(x,z)(Y+Z)(x+z) \, dz.
\]

Introducing the group $S(t)=e^{-it\theta (D)}$, we obtain

\begin{align*}
       &  Z  =  S(t)Z_0 -i\int_0^t S(t-\tau)\Big[ \big(w*V(\cdot,0)\big)Y - \int w(z)V(x,z)Y(x+z)\, dz \Big]d\tau\\
       & \qquad \qquad -i\int_0^tS(t-\tau)\Big[ \big(w*V(\cdot,0)\big)Z - \int w(z)V(x,z)Z(x+z)\, dz\Big]d\tau.
\end{align*}
This can be written under the form
\begin{equation} \label{id:Z-point-fixe}
Z= S(t)Z_0+L_1(V)+ L_2(V) +Q_1(Z,V)+Q_2(Z,V)
\end{equation}
where the linearised operators are
\begin{align*}
& L_1(V) = -i\int_0^t S(t-\tau)\big[ (w*V(\cdot,0))Y \big]d\tau,\\
&    L_2(V) = i\int_0^t S(t-\tau)\big[ \int dz w(z)V(x,z)Y(x+z) \big]d\tau,
\end{align*}
and the quadratic terms are
\begin{align*}
 &   Q_1(Z,V) = -i\int_0^t S(t-\tau)\big[ (w*V(\cdot,0))Z \big]d\tau,\\
  &  Q_2(Z,V) = i\int_0^t S(t-\tau)\big[ \int dz w(z)V(x,z)Z(x+z) \big]d\tau.
  \end{align*}
The perturbed correlation function $V$ is given by
\begin{align} \label{id:V-point-fixe}
V= & \; \E(\overline{Y(x+y)}S(t)Z_0(x)+\overline{S(t)Z_0(x+y)}Y(x))\\ \nonumber
& +L_3(V)+L_4(V)\\ \nonumber
& +Q_3(Z,V)+Q_4(Z,V) + \E(\bar Z(x+y) Z(x))
\end{align}
where the corresponding linearized operators and quadratic terms are for $k=3,4$,
\begin{eqnarray*}
    L_{k}(V)&=& \E[\overline{Y(x+y)}L_{k-2}(V)(x)]+\E[\overline{L_{k-2}(V)(x+y)}Y(x)],\\
    Q_{k}(Z,V)&=& \E[\overline{Y(x+y)}Q_{k-2}(Z,V)(x)]+\E[\overline{Q_{k-2}(Z,V)(x+y)}Y(x)].
\end{eqnarray*}

Combining \eqref{id:Z-point-fixe} and \eqref{id:V-point-fixe} we arrive at the following fixed point equation for $(Z,V)$
\begin{equation} \label{ptfixe2}
\begin{pmatrix} Z\\ V\end{pmatrix} = \mathcal A_{Z_0} \begin{pmatrix} Z\\ V\end{pmatrix} = \begin{pmatrix} \mathcal A^{(1)}_{Z_0}(Z,V))  \\
\mathcal A^{(2)}_{Z_0}(Z,V)
    \end{pmatrix} 
\end{equation}
where we have set
\begin{align*}
& \mathcal A^{(1)}_{Z_0}(Z,V)= S(t)Z_0+L_1(V)+L_2(V)+Q_1(Z,V)+Q_2(Z,V),\\
& \mathcal A^{(2)}_{Z_0}(Z,V) =
\E[\overline{Y(x+y)}S(t)Z_0(x)+\overline{S(t)Z_0(x+y)}Y(x)] \\
&\qquad \qquad \qquad+L_3(V)+L_4(V)+Q_3(Z,V)+Q_4(Z,V)+\E[\overline{Z(x+y)}Z(x)].
\end{align*}
We solve the fixed point equation \eqref{ptfixe2} via a contraction argument for the application $\mathcal A_{Z_0}$ in the following Banach spaces for $(Z,V)$:
\begin{align}\label{spaceZ}
 E_Z = \mathcal C(\R, L^2(\Omega,H^{s_c}(\R^d))) \cap L^p(\R,W^{s_c,p}(\R^d,L^2(\Omega)))\\ \nonumber
\cap L^{d+2}(\R\times \R^d,L^2(\Omega)) \cap L^4(\R, L^q(\R^d,L^2(\Omega))),\\\label{spaceV}
 E_V = \mathcal C(\R^d,L^{\frac{d+2}{2}}(\R\times \R^d))\cap \mathcal C(\R^d,L^2(\R, B_2^{-\frac12,s_c}(\R^d))))  
\end{align}
where $p = 2\frac{d+2}{d}$, $s_c = \frac{d}2 -1$, $q = \frac{4d}{d+1}$. We endow $E_V$ and $E_Z$ with the norms
\begin{align}\label{normZ}
\|\cdot\|_Z = \|\cdot \|_{L^\infty(\R, L^2(\Omega,H^s(\R^d)))} +  \|\cdot \|_{L^p(\R,W^{s_c,p}(\R^d,L^2(\Omega)))} \\\nonumber 
+ \|\cdot\|_{L^{d+2}(\R\times \R^d,L^2(\Omega))} + \|\cdot\|_{ L^4(\R, B_q^{0,\frac14}(\R^d,L^2(\Omega)))},\\\label{normV}
 \|\cdot\|_V = \|\cdot\|_{\mathcal{C}(\R^d,L^{\frac{d+2}{2}}(\R\times \R^d))} + \|\cdot\|_{\mathcal{C}(\R^d, L^2(\R, B_2^{-\frac12,s_c}(\R^d)))}.
\end{align}

\begin{remark} The spaces and norms we chose are driven by the following considerations. The regularity $s_c$ is the critical regularity of the cubic Schrödinger equation in dimension $d$. The choice of the Lebesgue exponents $p$ and $d+2$ for $Z$ and $\frac{d+2}{2}$ and $2$ for $V$ are the ones required to put $Q_1$ and $Q_2$ in the target space for the solution $\mathcal C(\R,L^2(\Omega,H^{s_c}(\R^d)))$. The regularity in the low frequencies for $V$, namely the $-\frac12$ in $B_2^{-\frac12,s_c}$ is due to a low frequencies singularity that we see appearing in Proposition \ref{prop:magicProposition}.

The choice of the $L^\infty$ norm in the variable $y$ is due to the fact that Lebesgue and Sobolev norms are invariant under the action of translations and thus the norms of $V$ in $x$ should be uniformly bounded in the variable $y$. 

\end{remark}

\section{Linear estimates}\label{sec:linear}

\subsection{Strichartz estimates}\label{subsec:Strichartz}

\begin{proposition}[Strichartz estimates] \label{pr:Strichartz-free-evolution}
Let $\theta$, $\tilde \theta$ defined as in \eqref{id:def-theta}, \eqref{id:def-tildetheta-theta0} such that 

\begin{enumerate}[i)]
    \item $\theta$ satisfies the ellipticity assumption \eqref{id:uniform-ellipticity-theta},
    \item $\tilde \theta \in C^{d+2}(\R^d) \cap  W^{d+2, \infty} (\R^d)$.
\end{enumerate}

Let $(p,q) \in [2, \infty]^2$ such that
\[
\frac 2p + \frac dq = \frac d2, \qquad (p,q) \ne (2, \infty) \text{ if } d=2.
\]
Then there exists a constant $C = C(p,q,d) $ such that for any $u_0 \in L^2(\R^d)$,
\[
\big \lVert e^{-it \theta(D)} u_0 \big \rVert_{L^p_t L^q_x} \le C \lVert u_0 \rVert_{L^2_x}.
\]
\end{proposition}

\begin{remark}

The regularity assumption ii) is satisfied, for example, in the case of a very short range interaction potential $\la y \ra^{d+2} w\in M$, or in the case of an integrable potential $w\in M$ and of a smooth distribution function $g\in W^{d+2,1}$. The uniform ellipticity is then satisfied, for example, if the interaction potential and the density are small enough $\| \la y \ra^2 w\|_{M}\| g\|_{L^1}<C(d)$ or the distribution function is spread enough $\| w\|_M\| \nabla^2 g\|_{L^1}<C(d)$, respectively.

The ellipticity \eqref{id:uniform-ellipticity-theta} can be false at high densities, no matter the interaction potential. Indeed, consider a fixed potential $w$ and for $\rho>0$ an equilibrium $g=\rho g^*$, with $g^*$ and $w$ both Schwartz and nonzero. We have $ \theta=\theta_0+|\xi|^2-(2\pi)^{d/2}\rho \hat w* g^*$. Since there exists a point at which the Hessian of $ \hat w* g$ is not nonnegative, for large $\rho$ the Hessian of $\theta$ is not positive definite at that point.

The failure of the ellipticity \eqref{id:uniform-ellipticity-theta} would lead to different linearized dynamics, and would strongly differ from the reduced Hartree equation \eqref{Cauchyprob-reduced}-\eqref{Cauchyprob-density-matrices-reduced} where this issue is absent.

\end{remark}

\begin{proof}
Let us divide the proof in steps. \\
\emph{Step 1}: localization in frequencies.\\
Let $\chi_0, \chi \in C_c^\infty(\R^d)$ be such that 
\begin{enumerate}[i)]
\item $0 \le \chi_0, \chi \le1$ and $\supp \chi_0 \subset \{ \lvert \xi \rvert \le1 \}$, $\supp \chi \subset \{ 1 \le \lvert \xi \rvert \le2 \}$;
\item for any $\xi \in \R^d$
\[
\chi_0 (\xi) + \sum_{\lambda \in 2^{\N}} \chi (\lambda^{-1} \xi) = 1;
\]
\item $\exists C \in (0,1)$ such that for any $\xi \in \R^d$
\[
c \le \chi_0^2 (\xi) + \sum_{\lambda \in 2^{\N}} \chi^2 (\lambda^{-1} \xi) \le 1.
\]
Let us call $\chi_\lambda (\xi) \coloneqq \chi (\lambda^{-1} \xi)$, $\lambda \in 2^{\N}$. We define the following frequency localised function
\[
u_\lambda(t,x) \coloneqq e^{- it \theta (D)} \chi_{\lambda} (D) u_0(x), \quad \lambda \in 2^{\N} \cup \{0\}
\]
which is given by
\[
\begin{split}
    u_\lambda(t,x) & = (2\pi)^{-d} \int_{\R^d} \Big[\int_{\R^d} e^{-it \theta(\xi) + i (x-y)\cdot \xi} \chi_{\lambda}(\xi) d\xi\Big] u_0(y) dy \\
    & = \int_{\R^d} I_\lambda (t, x-y) u_0(y)dy = [I_{\lambda} (t, \cdot) \ast u_0(\cdot) ] (x),
\end{split}
\]
where 
\[
I_\lambda(t,x) \coloneqq (2\pi)^{-d} \int_{\R^d} e^{-it\theta(\xi) + i x\cdot \xi} \chi_\lambda(\xi) d\xi.
\]
\end{enumerate}
\emph{Step 2}: frequency localised dispersive estimate.\\
We recall the following result, which is Theorem 3 in \cite{ABZ15}.
\begin{theorem}
\label{thm:3ABZ}
Consider
\[
I(\mu) = \int_{\R^d} e^{i\mu \phi(\xi)} \psi(\xi) d\xi,
\]
with $\phi \in C^{d+2}(\R^d)$, $\psi \in C^{d+1}_0 (\R^d)$. Let us define $K \coloneqq \supp \psi $. Assume
\begin{enumerate}[i)]
\item $\mathcal M_{d+2}(\phi) = \sum_{2 \le \lvert \alpha \rvert \le d+2} \sup_{\xi \in K_{\epsilon_0}} \lvert \nabla^{\otimes \alpha} \phi (\xi) \rvert < +\infty$ where $K_{\epsilon_0}$ is a neighbourhood of $K$, 
\item $\mathcal N_{d+1}(\psi) = \sum_{\lvert \alpha \rvert \le d+1} \sup_{\xi \in K} \lvert \nabla^{\otimes \alpha} \psi(\xi) \rvert < + \infty$,
\item $a_0 = \inf_{\xi \in K_{\epsilon_0}} \lvert \det \nabla^{\otimes 2} \phi(\xi) \rvert > 0$ and the map $\xi \mapsto \nabla \psi(\xi)$ is injective.
\end{enumerate}
Then there exists $C>0$ depending only on the dimension $d$ such that 
\begin{equation}\label{ineq:stat-phase}
\lvert I(\mu) \rvert \le \mu^{- \frac d2} C a_0^{-1} (1+ \mathcal M_{d+2}^{\frac d2} ) \mathcal N_{d+1}.
\end{equation}
\end{theorem}

Let 
\[
\phi (t,x,\xi) \coloneqq - \theta(\xi) + \frac{x}{t} \cdot \xi,
\]
so that 
\[
I_\lambda (t,x) = (2\pi)^{-d} \int_{\R^d} e^{it \phi(t,x,\xi)} \chi_\lambda (\xi) d\xi.
\]
We use \eqref{ineq:stat-phase} to estimate the $L^\infty_x$ norm of $I_\lambda$ for every $\lambda \in 2^{\N} \cup \{0\}$ uniformly wrt $\lambda$ and $x$. 
Let us show that $\phi_\lambda$, $\chi$ satisfy hypotheses of \Cref{thm:3ABZ} and that we can bound the RHS of \eqref{ineq:stat-phase} uniformly wrt $\lambda$;
\begin{enumerate}[i)]
\item from assumption $ii)$ we have $\mathcal M_{d+2} < C$ where $C$ does not depend on $\lambda$;
\item the boundness of $\mathcal N_{d+1}$ follows from definitions of $\chi_\lambda$, $\chi_0$ in the previous step;
\item since $d\ge3$, from assumption $ii)$ it follows that $\phi$ is at least $C^3$. Then $a_0 > 0$ iff
\[
\qquad \inf_{\xi \in \R^d} |\det \nabla^{\otimes 2} (- \tilde \theta(\xi))|  > 0.
\]
Then, by assumption $i)$, there exists $c_1 > 0$ such that \[
\inf_{\xi \in \R^d} \lvert \det \nabla^{\otimes 2} \phi (\xi) \rvert \ge c_1
\]
(uniformly wrt $\lambda$). Moreover,we show that the gradient of $\tilde \phi = - \phi$ is injective:
\begin{align}
\nabla \tilde \phi (\xi_1) - \tilde \nabla  \tilde \phi(\xi_2) & = \int_0^1 \frac{d}{ds} [ \nabla \tilde \phi_(s \xi_1 + (1-s)\xi_2) ] ds \\
& = \int_0^1 \nabla^{\otimes 2} \tilde \phi(s\xi_1 + (1-s) \xi_2) \cdot (\xi_1 - \xi_2) ds.
\end{align}
Then, by assumption $i)$ we get
\[
\lvert \nabla \tilde \phi(\xi_1) - \nabla \tilde \phi(\xi_2) \rvert \lvert \xi_1 - \xi_2 \rvert \ge \langle \nabla \tilde \phi(\xi_1) - \nabla \tilde \phi(\xi_2), \xi_1 - \xi_2 \rangle \ge c_1 \lvert \xi_1 - \xi_2 \rvert^2;
\]
\end{enumerate}
Moreover,  the same bound holds if $\lambda =0$. \\
Finally, combining the previous estimate with Young's inequality, we obtain
\begin{equation}
\label{eq:freq_loc_disp_est}
\begin{split}
\lVert e^{it\theta(D)} \chi_\lambda (D) u_0 \rVert_{L^\infty_x} & \le \lVert I_\lambda (t, \cdot) \rVert_{L^\infty_x} \lVert u_0 \rVert_{L^1_x} \\
& \le C t^{- \frac d2} \lVert u_0 \rVert_{L^1_x}
\end{split}
\end{equation}
where $C$ depends only on $d$. \\
\emph{Step 3}: Strichartz estimates.\\
We recall the following result, it is Theorem 1.2 in \cite{KT98}.
\begin{theorem}
\label{thm:KT}
Let $(X,dx)$ be a measure space and $H$ an Hilbert space. Suppose that for all $t \in \R$ an operator $U(t) \colon H \to L^2(x)$ obeys the following estimates
\begin{itemize}
    \item for all $t$ and $f \in H$ we have
    \[
    \lVert U(t) f \rVert_{L^2_x}\lesssim \lVert f \rVert_H;
    \]
    \item for all $t \ne s$ and $g \in L^1(X)$ 
    \[
    \lVert U(s) U^*(t) g \rVert _{L^\infty} \lesssim \lvert t-s\rvert^{-\sigma} \lVert g \rVert_{L^1},
    \]
    for some $\sigma >0$.
\end{itemize}
Then 
\[
\lVert U(t) f \rVert_{L^q_t L^r_x} \lesssim \lVert f \rVert_H,
\]
for all $(q,r) \in [2,+\infty]^2$ such that
\[
\frac1q + \frac{\sigma}r = \frac{\sigma}2, \quad (q,r,\sigma) \ne (2,\infty,1)
\]
where the endpoint $P = \big (2, \frac{2\sigma}{\sigma -1})$ is admissible if $\sigma >1$.
\end{theorem}
We observe that \Cref{thm:KT} holds with the choice $U(t) = e^{-it \theta(D)} \chi_\lambda (D) \colon L^2_x \to L^2_x$, $\sigma = \frac d2$. Then, we have that 
\begin{equation}
\label{eq:loc_stri}
\big \lVert e^{-it \theta(D)} \chi_\lambda (D) u_0 \big \rVert_{L^p_t L^q_x} \le C \lVert u_0 \rVert_{L^2_x},
\end{equation}
for all $(p,q) \in [2,+\infty]^2$ such that
\begin{equation}
    \label{eq:adm_ass}
\frac 2p + \frac dq = \frac d2, \qquad (p,q) \ne (2, \infty) \text{ if } d=2.
\end{equation}
Moreover, estimate \eqref{eq:loc_stri} holds the same if we replace the RHS by $\lVert \chi_\lambda(D) u_0 \rVert_{L^2_x}$. Indeed, we can replace $\chi_\lambda$ in \eqref{eq:loc_stri} by a function $\tilde \chi_\lambda \in C_c^\infty$ such that $\tilde \chi_\lambda \chi_\lambda = \chi_\lambda$, then
\[
\big \lVert e^{-it \theta(D)} \chi_\lambda (D) u_0 \big \rVert_{L^p_t L^q_x}  = \big \lVert e^{-it \theta(D)} \tilde \chi_\lambda \chi_\lambda (D) u_0 \big \rVert_{L^p_t L^q_x} \le C \lVert \chi_\lambda (D) u_0 \rVert_{L^2_x}.
\]
Then, since $[ e^{-it\theta(D)},\chi_\lambda(D) ]=0$, by Littlewood-Paley Theorem, Minkowski's inequality and \eqref{eq:loc_stri}, for all $(p,q)$ satisfying \eqref{eq:adm_ass}, $q < \infty$, we have the following
\[
\begin{split}
 \big  \lVert  e^{-it\theta(D)} u_0 \big \rVert_{L^p_t L^q_x} & \simeq \bigg \lVert \big \lVert e^{-it\theta(D)} \chi_0(D) u_0 \big \rVert_{L^q_x} + \Big \lVert \Big ( \sum_{\lambda \in 2^{\N}} \lvert e^{-it\theta(D)} \chi_\lambda (D) u_0 \rvert ^2 \Big )^{\frac 12} \Big \rVert_{L^q_x} \bigg \rVert_{L^p_t}\\
 & \le \big \lVert e^{-it\theta(D)} \chi_0(D) u_0 \big \rVert_{L^p_t L^q_x} + \bigg ( \sum_{\lambda \in 2^{\N}} \big \lVert e^{-it\theta(D)} \chi_\lambda(D) u_0 \big \rVert_{L^p_tL^q_x}^2 \bigg )^{\frac 12}\\
 & \lesssim \big \lVert \chi_0(D) u_0 \big \rVert_{L^2_x} + \bigg ( \sum_{\lambda \in 2^{\N}} \big \lVert \chi_\lambda (D) u_0 \big \rVert_{L^2_x}^2 \bigg)^{\frac12} \\
 & \simeq \lVert u_0 \rVert_{L^2_x}.
\end{split}
\]
\end{proof}

\begin{remark}
We observe that it would be possible to adapt the proof of Theorem 1.2 in \cite{OL23} in order to prove the Strichartz estimates with slightly different assumptions on $\theta$. That is, $\theta \in C^{n_0}(\R^d)$, $n_0 > \frac{d+2}2$ and $\langle \xi \rangle ^{n-2} \nabla^{\otimes n} \tilde \theta \in L^\infty$ for any $n=2, \dots, n_0$.
\end{remark}

Proposition \ref{pr:Strichartz-free-evolution} leads classically, see e.g.\ \cite{Tao06} section 2.3, to the following results: 

\begin{corollary}\label{cor:strichartz}
Let $(q_1,r_1)$ admissible and $(q_2,r_2,s_2)$ such that $r\geq 2$ where $\frac{1}{r}=\frac{1}{r_2}+\frac{s_2}{d}$ and $\frac{2}{q_2}+\frac{d}{r_2}+s_2=\frac{d}{2}$. There exists $C>0$ such that for all $u_0\in H^{s_2}$ and $F\in L_t^{q_1'}H_x^{s_2,r_1'}$: 
\begin{equation}\label{ineq:strichartz-derivatives}
    \Big \lVert S(t)u_0-i\int_0^t S(t-\tau)F(\tau)d\tau \Big \rVert_{L_t^{q_2}L_x^{r_2}}\leq C \Big[\lvert\lvert u_0 \rvert \rvert_{H_x^{s_2}} + \lvert\lvert F \rvert \rvert_{L_t^{q_1'}H_x^{s_2,r_1'}}\Big].
\end{equation}

\end{corollary}

\subsection{Representation formulas for the linear terms}\label{subsec:formulas}

We record here suitable expressions for the operators $L_1$, $L_2$, $L_3$ and $L_4$. We introduce the semi-group of operators 
$$
\mathcal T_\xi(t) U(x) = \frac{1}{(2\pi)^{\frac d2}} \int e^{-i(\theta(\eta+\xi)-\theta(\eta)-\theta(\xi))t}e^{i\eta x}\hat U(\eta)d\eta
$$

\begin{proposition}

One has the following formulas for $V$ a Schwartz function:
\begin{equation}
\label{id:formula-L1}L_1(V)  = -i \int \int_0^t f(\xi) e^{i\xi x-i\theta(\xi) t} \mathcal T_{\xi}(t-\tau) S(t-\tau)\big[ w*V(\cdot,0,\tau)\big] \, d\tau dW(\xi) ,
\end{equation}

\begin{equation} \label{id:formula-L2}
 L_2(V)  = i \int \int \int_0^t f(\xi)e^{i\xi x-i\theta(\xi)t} e^{i\xi z} w(z) \mathcal T_\xi(t-\tau)S (t-\tau) V(\cdot,z,\tau) \, d\tau dz dW(\xi),
\end{equation}

\begin{equation} \label{id:formula-L3}
L_3(V)=i \iint \int_0^t  ( g(\xi+\zeta)-g(\xi))e^{-i(\theta(\xi+\zeta)-\theta(\xi))(t-\tau)}e^{i\zeta x-i\xi y}\hat{w}(\zeta) \hat{V}(\zeta,0,\tau)\,  d\tau d\xi d\zeta,
\end{equation}
and
\begin{multline} \label{id:computation-L4V}
 L_4(V)  \\
= -i(2\pi)^{-\frac d2} \iiint  \int_0^t (g(\xi+\zeta )- g (\xi))  e^{-i(t-\tau)(\theta(\xi+\zeta)-\theta(\xi))}e^{i\zeta x+i\xi (z-y)}  w(z) \Hat{V}(\zeta,z,\tau) \, d\tau  \, dz \, d\xi d\zeta.
\end{multline}

\end{proposition}

\begin{remark}[Formal properties of the formulas]

The formulas \eqref{id:formula-L1} and \eqref{id:formula-L2} display a Galilei transformation-type effect. Indeed, the symbol of the group $\mathcal T_{\xi}(t)$ satisfies (due to the ellipticity condition \eqref{id:uniform-ellipticity-theta} for $\theta$)
$$
|\nabla_\eta (\theta(\xi+\eta)-\theta(\eta)-\theta(\xi))|\approx |\xi| \qquad \mbox{and} \qquad \frac{\xi}{|\xi|}.\nabla_\eta (\theta(\xi+\eta)-\theta(\eta)-\theta(\xi))\approx |\xi|
$$
so that $\mathcal T_{\xi}(t)$ formally corresponds to transport with velocity $\xi$. This is clear when $\theta=|\xi|^2$ in which case $T_\xi(t)$ is the space translation of vector $2\xi t$.

From the formulas \eqref{id:formula-L1} and \eqref{id:formula-L2}, one can then expect $L_1(V)$ and $L_2(V)$ to enjoy the same dispersive estimates as a solution to $i\partial_t u=\theta(D) u+V$.
This is obtained by noticing that $T_\xi (t)S(t)$ enjoys the same dispersive estimates as $S(t)$, and by formally discarding the effects of the extra variables $\xi$ and $z$.

These formulas also hint to the fact that $L_1(V)$ and $L_2(V)$ could enjoy improved dispersive estimates than that of the group $S(t)$ alone. Indeed, the operator $T_\xi(t)$ amounts to translating in the direction $\xi t$. When averaging over $\xi$ such transport effects in all directions, this should produce an additional damping mechanism. This is made rigorous in Proposition \ref{prop:magicProposition}.

\end{remark}

\begin{proof}

We first remark that from the definition \eqref{id:def-V} one has
$$
\overline{V(x,y)}=V(x+y,-y)
$$
which in Fourier translates into the relation
\begin{equation} \label{id:relation-hatV-barhatV}
    \overline{\Hat{V}(\eta,y)} = \Hat{V}(-\eta,-y)e^{-i\eta y}.
\end{equation}

\noindent \underline{Formula for $L_1$.} By the definition \eqref{equilibria} of the equilibrium we have
\begin{align*}
L_1(V) & = -i\int_0^t S(t-\tau)\big[ (w*V(\cdot,0))Y \big]d\tau,\\
& = -i \int \int_0^t f(\xi)e^{-i \theta(\xi)\tau }  S(t-\tau)\big[ (w*V(\cdot,0))e^{i\xi x} \big]d\tau dW(\xi) .
\end{align*}

We readily check that
\begin{equation} \label{id:formula-galilei-St}
e^{i \theta(\xi)t}S(t)(e^{i\xi x}U) = e^{i\xi x}  \mathcal T_\xi(t)S(t)U,
\end{equation}
which gives the desired identity \eqref{id:formula-L1}.

\medskip

\noindent \underline{Formula for $L_2$}. Using again the definition \eqref{equilibria} of the equilibrium we have
\begin{align*}
 L_2(V) & = -i\int_0^t S(t-\tau)\big[ \int dz w(z)V(x,z)Y(x+z) \big]d\tau,\\
 &= -i \int \int \int_0^t f(\xi)e^{-i\theta(\xi)\tau} e^{i\xi z} w(z) S(t-\tau)\big[ V(x,z)e^{i\xi x} \big]d\tau dz dW(\xi).
\end{align*}
One then obtains \eqref{id:formula-L2} by appealing to \eqref{id:formula-galilei-St}.

\medskip

\noindent\underline{Formula for $L_3$}. We decompose
\begin{align*}
L_3(V) & = \E[\overline{Y(x+y)}L_1(V)(x)] + \E[\overline{L_1(V)(x+y)}Y(x)]\\
&=L_3^{(1)}(V)+L_3^{(2)}(V).
\end{align*}
We notice that
$$
\overline{L_3^{(1)}(V)(x,y)}=L_3^{(2)}(V)(x+y,-y)
$$
which in Fourier gives
\begin{align}    
  \nonumber  \widehat{L_3^{(1)}(V)}(\eta,y)&= \frac{1}{(2\pi)^{\frac d2}}\int \overline{L_3^{(2)}(V)(x+y,-y)}e^{-i\eta x}dx\\
 \nonumber  & = \frac{1}{(2\pi)^{\frac d2}}\overline{\int L_3^{(2)}(V)(z,-y)e^{i\eta (z-y)}dx}\\
  \label{id:relation-L31-L32}  &= e^{i\eta y}\overline{\widehat{L_3^{(2)}}(- \eta,-y)}.
\end{align}

Hence it suffices to compute $ \widehat{L_3^{(1)}(V)}$ in order to retrieve $\widehat{L_3(V)}$ as then
\begin{equation} \label{id:computation-L3-decomposition-Fourier} 
\widehat{L_3(V)}(\zeta,y) = \widehat{L_3^{(1)}(V)}(\zeta,y)+e^{i\zeta y}\overline{\widehat{L_3^{(1)}(V)}(-\zeta,-y)}.
\end{equation}
We infer from \eqref{id:formula-L1} that
$$
L_1(V)  = - i \int \int \int_0^t f(\xi) e^{i\xi x-i\theta(\xi) t} e^{i\zeta x} e^{-i(t-\tau)(\theta(\zeta+\xi)-\theta(\xi))}  \hat w(\zeta) \hat V(\zeta,0,\tau) \, d\tau d\zeta dW(\xi) .
$$

Using \eqref{equilibria} yields
\begin{eqnarray*}
    L_3^{(1)}(x,y)&=& - i \int \int_0^t f^2(\xi) e^{-i\xi y} e^{i\zeta x} e^{-i(t-\tau)(\theta(\zeta+\xi)-\theta(\xi))}  \hat w(\zeta) \hat V(\zeta,0,\tau) \, d\tau d\zeta d \xi  .
\end{eqnarray*}
So we get: 
$$
\widehat{L_3^{(1)}(V)}(\zeta,y)=-i(2\pi)^{\frac d2} \hat{w}(\zeta)\int  \int_0^t  f^2(\xi) e^{-i(\theta(\xi+\zeta)-\theta(\xi))(t-\tau)}e^{-i\xi y} \hat{V}(\zeta,0,\tau) \, d\tau d\xi.
$$
Injecting the above identity in \eqref{id:computation-L3-decomposition-Fourier} and using \eqref{id:relation-hatV-barhatV} finally gives
\[
\widehat{L_3(V)}(\zeta,y)=-i(2\pi)^{\frac d2} \hat{w}(\zeta)\int_0^t  \hat{V}(\zeta,0)\int  ( f(\xi)^2- f(\xi+\zeta)^2)e^{-i(\theta(\xi+\zeta)-\theta(\xi))(t-\tau)}e^{-i\xi y}d \tau\ d\xi.
\]
This is \eqref{id:formula-L3}.

\medskip

\noindent \underline{Formula for $L_4$}. It is very similar to $L_3$. We first decompose
\begin{align*}
L_4 (V) & =\E[\overline{Y(x+y)}L_2(V)(x)] + \E[\overline{L_2(V)(x+y)}Y(x)] \\
&= L_4^{(1)}(V)+L_4^{(2)}(V).
\end{align*}
As for the analogue decomposition for $L_3$, we have
\begin{equation}       \label{id:relation-L41-L42}
 \widehat{L_4^{(1)}(V)}(\eta,y) = e^{i\eta y}\overline{\widehat{L_4^{(2)}}(- \eta,-y)}.
\end{equation}
Hence it suffices to compute $L_4^{(1)}$. We infer from \eqref{id:formula-L2} that
$$
L_2(V)(x)= \frac{i}{(2\pi)^{\frac d2}}\int_0^t \iiint f(\xi) w(z) \hat V(\eta-\xi,z,\tau)e^{i\eta x+i\xi z} e^{-i(t-\tau)\theta(\eta)-i\tau \theta(\xi)}  \, dz\, d\eta \, d\tau \, dW(\xi)
$$
so that using \eqref{equilibria} one computes that
$$
    L_4^{(1)}(x,y)= \frac{i}{(2\pi)^{\frac d2}}\int_0^t  \iiint  f^2 (\xi)w(z) \hat V(\eta-\xi,z,\tau)e^{i(\eta-\xi) x+i\xi (z-y)} e^{i(t-\tau) (\theta(\xi)-\theta(\eta))}  \, dz\, d\eta \, d\tau \, d\xi.
$$
One thus obtains, using a change of variables, that
$$
\widehat{L_4^{(1)}}(\zeta,y)= i \int_0^t  \iint  f^2 (\xi)w(z) \hat V(\zeta,z,\tau)e^{i\xi (z-y)} e^{-i(t-\tau) (\theta(\xi + \zeta)-\theta(\xi))}  \, dz \, d\tau \, d\xi.
$$
Using successively \eqref{id:relation-L41-L42} and \eqref{id:relation-hatV-barhatV}, and then changing variables
\begin{align*}
\widehat{L_4^{(2)}}(\zeta,y) & =e^{i\zeta y}\overline{\widehat{L_4^{(1)}}(-\zeta,-y)} \\
& =-ie^{i\zeta y} \int_0^t  \iint  f^2 (\xi)w(z) \overline{\hat V(-\zeta,z,\tau)}e^{-i\xi (z+y)} e^{-i(t-\tau) (\theta(\xi)-\theta(\xi-\zeta))}  \, dz \, d\tau \, d\xi\\
& =-i  \int_0^t  \iint  f^2 (\xi)w(z) \hat V(\zeta,-z,\tau) e^{i(\zeta-\xi) (z+y)} e^{-i(t-\tau) (\theta(\xi)-\theta(\xi-\zeta))}  \, dz \, d\tau \, d\xi \\
& =-i  \int_0^t  \iint  f^2 (\xi+\zeta)w(z) \hat V(\zeta,z,\tau) e^{i \xi (z-y)} e^{-i(t-\tau) (\theta(\zeta+\xi)-\theta(\xi))}  \, dz \, d\tau \, d\xi .
\end{align*}
Combining the two identities above concludes the proof of \eqref{id:computation-L4V}.

\end{proof}

\subsection{The issue of low frequency regularity in the linear response}

\begin{proposition}\label{prop:magicProposition} Let $\sigma,\sigma_1 \geq 0$, $\sigma_1 < \frac{d}{2}$, $p_1>2$, $q_1\geq 2$ such that
\[
\frac2{p_1} + \frac{d}{q_1} = \frac{d}{2} -\sigma_1.
\]
Assuming that $\an{\xi}^{2\lceil \sigma \rceil} g \in W^{2,1}$ and $\tilde \theta \in W^{4,1}$ along with the ellipticity assumption, there exists a constant $C_\theta$ (decreasing with $\lambda_*$ and increasing with $\|\tilde \theta\|_{W^{4,1}}$) such that for all $U \in L^2_t,B_2^{-1/2 + \sigma_1,\sigma+\sigma_1-1/2}$,

\begin{equation}\label{ineq:magicProposition1.1}
    \big\|\int_0^\infty S(t-\tau) [U(\tau)Y(\tau)] d\tau \big\|_{L^{p_1}(\R,W^{\sigma,q_1}(\R^d,L^2(\Omega)))} \leq C_\theta \|\an{\xi}^{2\lceil \sigma \rceil} g \|_{W^{2,1}} \|U\|_{L^2_t,B_2^{-1/2+\sigma_1,\sigma+\sigma_1-1/2}}.
\end{equation}


\end{proposition}

\begin{proof}
    We start by taking $U$ in the Schwartz class to give a sense to the computations and we conclude by density. 
    
Set $L_1^\infty(U):=\int_0^\infty S(t-\tau)  Y(\tau) U(\tau)d\tau$. 
    
We denote for $\eta\in \R^d$, $t\in \R$, $S_\eta(t)$,  the Fourier multiplier by
\[
\xi \mapsto e^{-it(|\xi|^2 + 2 \eta \cdot \xi + (2\pi)^d (-1)^\iota \hat{w}*g(\xi+\eta)-\hat{w}*g(\eta))} = e^{it(\theta(\eta) - \theta(\xi+\eta))}.
\]

    \textbf{Step 1:} We compute $\E[\lvert L_1^\infty(U)\rvert^2]$.

We have the commutation relation 
\[
S(t)(e^{i\eta x}U)= \mathcal F^{-1}\Big( e^{-it\theta(\xi)} \hat U (\xi-\eta)\Big) = e^{i\eta x} \mathcal F^{-1}\Big( e^{-it\theta(\xi+\eta)}\hat U(\xi)\Big)
\]
and because the Fourier transform is taken only on the space variable, we get
\[
S(t)(e^{i\eta x}U) = e^{i\eta x - it\theta(\eta)} \mathcal F^{-1}\Big( e^{-it(\theta(\xi+\eta) - \theta(\eta))}\hat U(\xi)\Big).
\]
We recognize
\[
S(t)(e^{i\eta x}U) = e^{i\eta x - it\theta(\eta)} S_\eta(t)U.
\]

We deduce
    \begin{eqnarray*}
        S(t-\tau)(U(\tau) Y(\tau))&=&\int f(\eta)e^{-i\tau\theta(\eta)}S(t-\tau)(e^{i\eta x}U(\tau))dW(\eta)\\
        &=&\int f(\eta)e^{i\eta x}e^{-it\theta(\eta)}S_\eta(t-\tau)U(s)dW(\eta).
    \end{eqnarray*}
    
    Then we get, using the definition of Wiener integral: 
\[
        \E[\lvert L_1^\infty(U)\rvert^2]=\int_{\eta\in \R^d}g(\eta)\bigg\lvert \int_0^\infty S_\eta(t-\tau)U(\tau)d\tau \bigg\rvert^2d\eta. 
\]

    This concludes Step 1.
    
    \textbf{Step 2:} We claim that: 
\[
\lvert\lvert L_1^\infty(U)\rvert\rvert_{L_t^{p_1},L_x^{q_1},L_\omega^2}\leq C(\theta) \|g\|_{W^{2,1}} \lvert \lvert U\rvert\rvert_{L_t^2,B_2^{\sigma_1-\frac{1}{2},\sigma_1-\frac{1}{2}}}
\]
where $C(\theta)$ is a constant depending on $\hat w*g$.

    Using step 1 and Minkowski inequality we have: 

\[
\lvert\lvert L_1^\infty(U)\rvert\rvert_{L_t^{p_1},L_x^{q_1},L_\omega^2}^2\leq \int_{\eta\in\R^d}g(\eta)\bigg\lvert\bigg\lvert \int_0^\infty S_{\eta}(t-\tau)U(\tau)d\tau\bigg\rvert\bigg\rvert_{L_t^{p_1},L_x^{q_1}}^2d\eta. 
\]

    By Strichartz's inequality, we have:
    \[
    \lvert\lvert L_1^\infty(U)\rvert\rvert_{L_t^{p_1},L_x^{q_1},L_\omega^2}^2\leq \int_{\eta\in\R^d}g(\eta)\bigg\lvert\bigg\lvert\int_0^\infty S_{\eta}(-\tau)U(\tau)d\tau\bigg\rvert\bigg\rvert_{B_2^{\sigma_1,\sigma_1}}^2d\eta.
    \]

    We introduce the variable $U_1$ defined by $\hat{U_1}(\xi)=\lvert\xi\rvert^{\sigma_1}\hat{U}(\xi)$ and we have, by Parseval's identity: 
\begin{multline*}
\lvert\lvert L_1^\infty(U)\rvert\rvert_{L_t^{p_1},L_x^{q_1},L_\omega^2}^2
\\ 
\leq \int_{\eta\in\R^d}g(\eta)\int_{\xi\in\R^d}\int_0^\infty\int_0^\infty e^{i(t_1-t_2)(\xi^2-2\xi.\eta - \tilde \theta_\xi(\eta))}\hat{U_1}(t_1,\xi)\overline{\hat{U_1}(t_2,\xi)}dt_2dt_1d\xi d\eta
\end{multline*}
where $\tilde \theta_\xi (\eta) = \tilde \theta(\xi + \eta) - \tilde \theta(\eta)$.

    We perform the change of variable $\eta_\xi=\eta- \tilde \theta_\xi(\eta)\frac{\xi}{2\lvert \xi\rvert^2}$. It is a $C^1$-diffeomorphism. Indeed, we have that the Jacobian matrix of $\eta\mapsto \eta_\xi$ is the identity minus the Jacobian matrix of $\eta \mapsto \theta_\xi(\eta)\frac{\xi}{2\lvert \xi\rvert^2}$. This last matrix is of rank $1$ and writes $\frac{\xi}{2|\xi|} \nabla_\eta \frac{\tilde \theta_\xi}{|\xi|}$. We deduce that the Jacobian is invertible if $1-\an{\frac{\xi}{2|\xi|},\nabla_\eta \frac{\tilde \theta_\xi}{|\xi|} }$ does not vanish. By definition of $\tilde \theta_\xi$, we have 
    \[
    \nabla \frac{\tilde \theta_\xi}{|\xi|} = \int_{0}^1 \nabla^{\otimes 2} \tilde \theta (\eta + t \xi) \frac{\xi}{|\xi|} dt.
    \]
Since $\nabla^{\otimes 2} \theta = 2 Id + \nabla^{\otimes 2} \tilde \theta$, we deduce that 
    \[
    1-\an{\frac{\xi}{2|\xi|},\nabla_\eta \frac{\tilde \theta_\xi}{|\xi|} } = \frac12 \int_{0}^1 \frac{\xi^T}{|\xi|} \nabla^{\otimes 2} \theta (\eta + t \xi)\frac{\xi}{|\xi|} dt > \frac12 \lambda_* .
    \]
 We denote $\phi_\xi=\Tilde{\eta_\xi}^{-1}$. This gives, by also doing the change of variable $t=t_2-t_1$:
\begin{multline*}
\lvert\lvert L_1^\infty(U)\rvert\rvert_{L_t^{p_1},L_x^{q_1},L_\omega^2}^2 \\
\leq \int_{\eta\in\R^d} \int_{\xi\in\R^d}  g(\phi_\xi(\eta)) jac(\phi_\xi(\eta))\int_{\R}\int_{D_t} e^{-it(\xi^2-2\xi\eta)}\hat{U_1}(t_1,\xi)\overline{\hat{U_1}(t+t_1,\xi)}dt_1dtd\xi d\eta,
\end{multline*}
where $D_t=[-t,\infty]$ and $jac(\phi_\xi(\eta))$ stands for the Jacobian matrix of $\eta\mapsto \phi_\xi(\eta)$.

    We integrate over $\eta$ to get: 
\begin{multline*}
\lvert\lvert L_1^\infty(U)\rvert\rvert_{L_t^{p_1},L_x^{q_1},L_\omega^2}^2\\
\leq (2\pi)^{d/2} \int_{\xi\in\R^d} \int_{\R}\int_{D_t} \mathcal{F}_\eta\Big(g(\phi_\xi(\eta)) jac(\phi_\xi(\eta))\Big)(-2\xi t) e^{-it\xi^2}\hat{U_1}(t_1,\xi)\overline{\hat{U_1}(t+t_1,\xi)}dt_1dtd\xi. 
\end{multline*}

    We use Cauchy-Schwarz inequality over $t_1$ to get: 
\[
\lvert\lvert L_1^\infty(U)\rvert\rvert_{L_t^{p_1},L_x^{q_1},L_\omega^2}^2\leq \int_{\xi\in\R^d}\int_{\R} \Big\lvert\mathcal{F}_\eta\Big(g(\phi_\xi(\eta)) jac(\phi_\xi(\eta))\Big)(-2\xi t)\Big\rvert \lvert \lvert \hat{U_1}(t_1,\xi)\rvert\rvert_{L_{t_1}^2}^2 dtd\xi. 
\]

    Then, we get by doing the change of variable $\tau=t\lvert\xi\rvert$: 
\[
\lvert\lvert L_1^\infty(U)\rvert\rvert_{L_t^{p_1},L_x^{q_1},L_\omega^2}^2\lesssim \int_{\xi\in\R^d} \int_{\R}\Big\lvert \mathcal{F}_\eta\Big(g(\phi_\xi(\eta)) jac(\phi_\xi(\eta))\Big)(-2\tau \frac{\xi}{|\xi|})\Big\rvert \lvert \xi\rvert^{-1}\lvert \lvert \hat{U_1}(t_1,\xi)\rvert\rvert_{L_{t_1}^2}^2 d\tau d\xi. 
\]

We claim that $(\tau,\xi) \mapsto \mathcal{F}_\eta\Big(g(\phi_\xi(\eta)) jac(\phi_\xi(\eta))\Big)(-2\tau \frac{\xi}{|\xi|}) $ belongs to $L^\infty_\xi(\R^d,L^1_\tau(\R))$. It is sufficient for this to prove that
\[
(\xi,\eta)\mapsto g(\phi_\xi(\eta)) jac(\phi_\xi(\eta))
\]
belongs to $L^\infty_\xi(\R^d,W^{2,1}_\eta(\R^d))$. This is implied by the fact that $g \in W^{2,1}(\R^d)$ and $\phi_\xi \in L^\infty_\xi W^{3,\infty}$, the latter coming from the fact that $\tilde \theta\in W^{4,1}$ which implies that $\tilde \theta_\xi \in L^\infty_\xi(\R^d,W^{3}_\eta(\R^d))$. We therefore get
\[
\lvert\lvert L_1^\infty(U)\rvert\rvert_{L_t^{p_1},L_x^{q_1},L_\omega^2}^2\leq C_\theta\|g\|_{W^{2,1}} \int_{\xi\in\R^d} \lvert \lvert \lvert \xi\rvert^{\sigma_1-\frac{1}{2}}\hat{U}(t_1,\xi)\rvert\rvert_{L_{t_1}^2}^2 dtd\xi=C_\theta\|g\|_{W^{2,1}} \lvert\lvert U\rvert\rvert_{L_t^2,B_2^{\sigma_1-\frac{1}{2},\sigma_1-\frac{1}{2}}}.
\]

    \textbf{Step 3:} We first suppose that $\sigma\in \N$.
    
    For $\alpha\in\N^d$ we write: $\lvert\alpha\rvert=\sum_{j=1}^d\alpha_j$ and $\partial^\alpha=\prod_{j=1}^d\partial^{\alpha_j}$.
    
    For $\eta\in\R^d$ we write $\eta^\alpha=\prod_{j=1}^d\eta_j^{\alpha_j}$.
    
    We have for any $\alpha\in\N^d$: 
    
    $$\partial^\alpha L_1^\infty(U)=\sum_{\gamma+\beta=\alpha}C(\alpha,\beta)\int_0^tS(t-s)\partial^\beta Y(s)\partial^\gamma U(s) ds.$$
    
    Indeed, $Y$ is almost everywhere differentiable and there holds, for $\lvert \beta\rvert\leq \lceil s\rceil$: $$\partial^\beta Y(s)=\int_{\R^d}i^{\lvert \beta\rvert}\eta^\beta f(\eta)e^{-is\theta(\eta)+i\eta x}dW(\eta). $$
    
    Replacing $Y$ by $\partial^\beta Y$ consists in replacing $f(\eta)$ by $i^{\lvert \beta\rvert}\eta^\beta f(\eta)$ while leaving $\theta$ unchanged. 
    
    Thus, using step 2 we get: 
    
    \begin{equation}\label{Str1}
        \lvert\lvert L_1^\infty(U)\rvert\rvert_{L_t^{p_1},W_x^{\sigma,q_1},L_\omega^2}^2\lesssim \lvert\lvert U\rvert\rvert_{L_t^2,B_2^{\sigma_1-\frac{1}{2},\sigma_1+\sigma-\frac{1}{2}}},
    \end{equation}
    
    where the constants depends on $\underset{\lvert \alpha\rvert\leq \sigma}{sup}\Big\lvert\Big\lvert \mathcal{F}_\eta\Big( \phi_\xi(\eta)^{2\alpha} g(\phi_\xi(\eta)) jac(\phi_\xi(\eta))\Big)\Big\rvert\Big\rvert_{L^1}.$

    We conclude by interpolation that inequalities (\ref{Str1})  holds for any $\sigma \geq 0$.

\end{proof}

\begin{corollary}\label{cor:CKtoMP} With the notations of Proposition \ref{prop:magicProposition}, we have that
\[
L_1^{\infty, 0} : U \mapsto \int_{0}^t S(t-\tau) [U(\tau) Y(\tau)]d\tau, \quad L_1^{\infty,\infty} : U \mapsto \int_{t}^\infty S(t-\tau) [U(\tau) Y(\tau)]d\tau
\]
are continuous operators from $L^2_t B_2^{-1/2+\sigma_1,-1/2+\sigma_2 +\sigma}$ to $L^{p_1}_t W^{\sigma,q_1}_x$.

\end{corollary}

\begin{proof} This is a standard application of the Christ-Kiselev lemma.
\end{proof}

\begin{corollary}\label{corollary:l1}
The operator 
\[
L_1\ : V \ \mapsto L_1(V)= -i\int_0^t S(t-\tau)\big[ w*V(\cdot,0)Y \big]d\tau 
\]
is continuous from $L^\infty_y L^2_t B_2^{-1/2,s_c}$ to $E_Z$ and thus from $E_V$ to $E_Z$.
\end{corollary}

\begin{proof}
    
    
We apply Corollary \ref{cor:CKtoMP} to $(p_1,q_1,\sigma_1,\sigma)$ equal to either 
\[
(p,p,0,s_c),\, (d+2,d+2,s_c,0),\, (4,q,\frac{d-3}{4},0) \text{or }(\infty,2,0,s_c) 
\]
to get the result.

    
    
    

\end{proof}

\begin{corollary}\label{corollary:l2}
The operator 
\[
L_2\ : V \ \mapsto L_2(V)= i\int_0^t S(t-\tau)\big[ \int dz w(z)V(x,z)Y(x+z) \big]d\tau 
\]
is continuous from $L^\infty_y L^2_t B_2^{-1/2,s_c}$ to $E_Z$ and thus from $E_V$ to $E_Z$.

\end{corollary}

\begin{proof}
    We begin by writing that, as $w$ is a finite measure:
    
    $$\lvert\lvert L_2(V)\rvert\rvert_{E_Z}\lesssim \int_{z\in \R^d}\bigg\lvert\bigg\lvert \int_0^t S(t-\tau)V(\cdot,z)T_zYd\tau\bigg\rvert\bigg\rvert_{E_Z}|w|(z)dz. $$
    
    For $z\in \R^d$, we have: 
    
    $$\bigg\lvert\bigg\lvert \int_0^t S(t-\tau)V(\cdot,z)T_zYd\tau\bigg\rvert\bigg\rvert_{E_Z}=\bigg\lvert\bigg\lvert \int_0^t S(t-\tau)T_{-z}V(\cdot,z)Yd\tau\bigg\rvert\bigg\rvert_{E_Z}.$$
    
    By applying Corollary \ref{cor:CKtoMP} as in Corollary \ref{corollary:l1} and using the invariance of the $B_2^{-1/2,s_c}$ norm under translations, we get: 
    
    $$\bigg\lvert\bigg\lvert \int_0^t S(t-\tau)T_{-z}V(\cdot,z)Yd\tau\bigg\rvert\bigg\rvert_{E_Z}\lesssim \lvert\lvert T_{-z}V(\cdot,z)\rvert\rvert_{L_t^2,B_2^{-\frac{1}{2},s_c}}=\lvert\lvert V(\cdot,z)\rvert\rvert_{L_t^2,B_2^{-\frac{1}{2},s_c}},$$
    
    and finally: 
    
    $$\bigg\lvert\bigg\lvert \int_0^t S(t-\tau)T_{-z}V(\cdot,z)Yd\tau\bigg\rvert\bigg\rvert_{E_Z}\lesssim \lvert\lvert V\rvert\rvert_{E_V}$$
    
    Then, we can conclude by writing that: 
    
    $$\lvert\lvert L_2(V)\rvert\rvert_{E_Z}\lesssim \lvert\lvert w \rvert\rvert_{M}\lvert\lvert V\rvert\rvert_{E_V}.$$
\end{proof}

\subsection{Estimates and invertibility on the last linear terms}

\begin{proposition} \label{pr:continuity-L4}

The operators $L_3$ and $L_4$ are continuous on $\mathcal C_y L^2_{t,x}\cap L^\infty_yL^2_{t,x}$ with for $k=3,4$,
\begin{equation} \label{bd:L4}
\| L_k (V)\|_{L^\infty_y L^2_{t,x}}\leq C_{\theta,w,g} \| V\|_{L^\infty_y L^2_{t,x}}.
\end{equation}
Moreover, we have that $L_4(V)$ and $L_3 (V)$ belong to $\mathcal C_y L^2_{t,x}$.

The constant is of the form $C_{\theta,w,g}=C_\theta \| \langle y \rangle w\|_{L^1} \| \nabla g\|_{W^{2,1}}$ for a constant $C_\theta$ that is decreasing with $\lambda_*$ and increasing with $\| \nabla^{\otimes 3}\Tilde \theta \|_{L^{\infty}}$.

\end{proposition}

In order to prove Proposition \ref{pr:continuity-L4}, we need the following lemma.

\begin{lemma}[Estimate for $\frac{e^{i\eta \alpha}}{\alpha}$-like principal values]

Consider a function $\tilde \Omega \in C^2(\mathbb R)$ satisfying $\tilde \Omega '>\lambda $ for some $\lambda>0$ and $\tilde \Omega''\in L^{\infty}$, then for any $q\in (1,\infty)$, for any $u\in W^{1,q}(\mathbb R)$ and $\eta\in \mathbb R$ there holds
\begin{equation} \label{bd:principal-value}
\left| p.v. \int_{\mathbb R} u(\alpha) \frac{e^{i\eta \alpha}}{\tilde \Omega(\alpha)}  d\alpha \right| \leq \frac{C}{\lambda}\Big(1+\frac{\|\tilde \Omega''\|_{L^\infty}}{\lambda}\Big) \| u\|_{W^{1,q}(\mathbb R)}
\end{equation}
for some universal $C$ that is independent of $\tilde \Omega $ and $\eta$ and where $P.v.$ stands for principal value.

\end{lemma}

\begin{remark}
Note that the estimate is false at the endpoint cases $q=1$ and $q=\infty$. To see it, it suffices to consider the Hilbert transform $ \frac{e^{i\eta \alpha}}{\tilde \Omega(\alpha)} =\frac{1}{\alpha}$. The Hilbert transform is ill-defined on $L^\infty$ which invalidates the estimate for $q=\infty$. Moreover, if $u\in W^{1,1}$, then $(\frac{1}{\alpha}*u)'\in L^{1,w}$ where $L^{1,w}$ is the weak $L^1$ space. However, there exists unbounded functions $v$ such that $v'\in L^{1,w}$, for example the $\log$ function.
The failure of the estimate in the case $q=1$ is the reason why $g\in W^{3,1}$ is required in the previous proposition.
\end{remark}

\begin{proof}

Up to dividing $\tilde \Omega$ by $\lambda$, we assume $\lambda=1$ without loss of generality. Then as $\tilde \Omega '>1$ there exists a unique zero $\alpha_0$ of $\tilde \Omega$. We decompose, omitting the $p.v.$ notation for simplicity,
\begin{equation}\label{id:principal-value-decomposition}
\int_{\mathbb R} \frac{e^{i\eta \alpha}}{\tilde \Omega(\alpha)} u(\alpha)d\alpha = I +II+III
\end{equation}
with
\begin{align*}
I  &=  \int_{|\alpha-\alpha_0|> 1} \left(\frac{e^{i\eta \alpha}}{\tilde \Omega(\alpha)} - \frac{e^{i\eta \alpha}}{\tilde \Omega'(\alpha_0)(\alpha-\alpha_0)} \right)u(\alpha)d\alpha , \\
 II & = \int_{\mathbb R} \frac{e^{i\eta \alpha}}{\tilde \Omega'(\alpha_0)(\alpha-\alpha_0)} u(\alpha)d\alpha , \\
III & = \int_{|\alpha-\alpha_0|<1} \left(\frac{e^{i\eta \alpha}}{\tilde \Omega(\alpha)} - \frac{e^{i\eta \alpha}}{\tilde \Omega'(\alpha_0)(\alpha-\alpha_0)} \right)u(\alpha) d\alpha  .
\end{align*}
To estimate $I$, note that $\tilde \Omega'(\alpha_0)>1$, and then that $|\tilde \Omega(\alpha)|> |\alpha-\alpha_0|$, so that by the H\"older inequality,
\begin{equation} \label{bd:principal-value-1}
|I|\leq 2 \int_{|\alpha-\alpha_0|\geq 1} \frac{|u(\alpha)|}{|\alpha-\alpha_0|}d\alpha \lesssim \| u\|_{L^q}.
\end{equation}
Next we have $II=\frac{e^{i \eta \alpha_0}}{\tilde \Omega'(\alpha_0)}v(\alpha_0)$ with $v= (\alpha \mapsto \frac{e^{i\eta \alpha}}{\alpha})* u$. Since
$$
|v(\tilde \alpha)|= \left|\int_{\mathbb R} \frac{e^{i\eta (\alpha- \tilde \alpha)}}{\alpha-\tilde \alpha} u(\alpha)d\alpha \right|  \leq \left|\int_{\mathbb R} \frac{1}{\alpha-\tilde \alpha} e^{i\eta \alpha}u(\alpha)d\alpha \right| 
$$
and
$$
|v'(\tilde \alpha)|= \left|\int_{\mathbb R} \frac{e^{i\eta (\alpha- \tilde \alpha)}}{\alpha-\tilde \alpha} u'(\alpha)d\alpha \right|    \leq \left|\int_{\mathbb R} \frac{1}{\alpha-\tilde \alpha} e^{i\eta \alpha}u'(\alpha)d\alpha \right|,
$$
the boundedness of the Hilbert transform on $L^q$ implies that $v\in W^{1,q}$ with $\|v\|_{W^{1,q}}\lesssim \| u\|_{W^{1,q}}$. By the Sobolev embedding one therefore has $v\in L^\infty$ so that
\begin{equation} \label{bd:principal-value-2}
|II|\lesssim \| u\|_{W^{1,q}}.
\end{equation}
Finally, to bound $III$ we note that $2 |\tilde \Omega(\alpha) -\tilde \Omega'(\alpha_0)(\alpha-\alpha_0)|\leq  \| \tilde \Omega''\|_{L^{\infty}} |\alpha-\alpha_0|^2$ by Taylor's formula. Hence if $|\alpha-\alpha_0|< \frac{|\tilde \Omega'(\alpha_0)|}{\| \tilde \Omega ''\|_{L^\infty}}$ (or for all $\alpha\in \mathbb R$ if $\| \tilde \Omega ''\|_{L^\infty}=0$) then we have
$$
\frac{1}{\tilde \Omega(\alpha)}=\frac{1}{\tilde \Omega'(\alpha_0)(\alpha-\alpha_0)}\frac{1}{1+\frac{\tilde \Omega(\alpha) -\tilde \Omega'(\alpha_0)(\alpha-\alpha_0)}{\tilde \Omega'(\alpha_0)(\alpha-\alpha_0)}}=\frac{1}{\tilde \Omega'(\alpha_0)(\alpha-\alpha_0)}+O(\|\tilde \Omega''\|_{L^\infty})
$$
where we used $\tilde \Omega'(\alpha_0)>1$. We recall $|\tilde \Omega (\alpha)|>|\alpha-\alpha_0|$ as $\tilde \Omega'>1$. If $|\alpha-\alpha_0|> \frac{|\tilde \Omega'(\alpha_0)|}{\| \tilde \Omega ''\|_{L^\infty}}$ there then holds $\frac{1}{|\Omega(\alpha)|}< \| \tilde \Omega'' \|_{L^\infty}$. Combining, we get that
\begin{align*}
\frac{e^{i\eta \alpha}}{\tilde \Omega(\alpha)} - \frac{e^{i\eta \alpha}}{\tilde \Omega'(\alpha_0)(\alpha-\alpha_0)} =O(\| \tilde \Omega''\|_{L^\infty})
\end{align*}
for all $|\alpha-\alpha_0|<1$. Thus, by the H\"older inequality one gets
\begin{equation} \label{bd:principal-value-3}
|III | \lesssim \| \tilde \Omega''\|_{L^\infty} \int_{|\alpha-\alpha_0|\leq 1} |u(\alpha)|d\alpha \lesssim \| \tilde \Omega''\|_{L^\infty}  \| u\|_{L^q}.
\end{equation}
Injecting the bounds \eqref{bd:principal-value-1}, \eqref{bd:principal-value-2} and \eqref{bd:principal-value-3} in \eqref{id:principal-value-decomposition} shows the desired result \eqref{bd:principal-value}.

\end{proof}

We can now turn to the proof of Proposition \ref{pr:continuity-L4}.

\begin{proof}[Proof of Proposition \ref{pr:continuity-L4}]

In what follows $\lesssim $ denotes inequalities where the implicit constant solely depend on $\lambda_*$ and $\| \theta \|_{W^{3,\infty}}$. The letter $\omega$ will denote the dual Fourier variable of $t$, and not the probability variable as in the rest of the article. This should create no confusion as no probability variable is involved in the proof. We introduce $U(x,z,t)=w(z)V(x,z,t)$, so that
$$
\widehat{L_4(V)}(\zeta,y,t)= - i\int_0^t \iint (g(\xi+\zeta)-g(\xi))e^{-i(t-\tau)(\theta(\xi+\zeta)-\theta(\xi))}\hat U(\zeta,z,\tau)e^{i\xi(z-y)} dzd\xi d\tau.
$$
We write the above under the form
$$
\widehat{L_4(V)}(\zeta,y,t)=-i(2\pi)^{\frac d2} \int (g(\xi+\zeta)-g(\xi))e^{-i\xi y}\big(\mathcal F_{x,y} U(\zeta,-\xi,\cdot)* (\mathbf{ 1}(\cdot \geq 0) e^{-i(\theta(\xi+\zeta)-\theta(\xi))\cdot })\Big) (t)  d\xi
$$
where we used that $U(t)=0$ for $t<0$. In a similar way, we have
\begin{multline*}
\widehat{L_3(V)}(\zeta,y,t)\\
=i(2\pi)^{\frac d2} \int (g(\xi+\zeta)-g(\xi))e^{-i\xi y}\big(\hat w(\zeta)\mathcal F_{x} V(\zeta,0,\cdot)* (\mathbf{ 1}(\cdot \geq 0) e^{-i(\theta(\xi+\zeta)-\theta(\xi))\cdot })\Big) (t)  d\xi .
\end{multline*}
We introduce
\[
U_4(\zeta,\xi,\omega) = \mathcal F_{x,y,t} U (\zeta,-\xi,\omega), \quad U_3 (\zeta,\xi, \omega) = \hat w (\zeta) \mathcal F_{x,t} V (\zeta, 0, \omega).
\]
Note that $U_3$ does not depend on $\xi$ but it unifies notations.

We recall that
$$
\mathcal F_t(\mathbf{ 1}(\cdot \geq 0) )(\omega)=\frac{1}{2}\sqrt{2\pi}\delta(\omega)+\frac{1}{\sqrt{2\pi}}\frac{1}{i\omega}.
$$
Therefore, applying the Fourier transform in time we have for $k=3,4$,
\begin{equation} \label{id:decomposition-L4}
L_k(V)=(-1)^{k+1}\frac{i(2\pi)^{\frac{d+1}{2}}}{2} L_{k,\delta}(V)+(-1)^{k+1}(2\pi)^{\frac{d-1}{2}}L_{k,p.v}(V)
\end{equation}
where the Dirac part is
$$
\mathcal F_{x,t}L_{k,\delta}(V)(\zeta,y,\omega)= \int (g(\xi+\zeta)-g(\xi))e^{-i\xi y}  U_k(\zeta,\xi,\omega)\delta(\omega+\theta(\xi+\zeta)-\theta(\xi))  d\xi
$$
and the principal value part is
$$
\mathcal F_{x,t}L_{k,p.v.}(V)(\zeta,y,\omega)= \int (g(\xi+\zeta)-g(\xi))e^{-i\xi y}  U_k(\zeta,\xi,\omega)\frac{1}{\omega+\theta(\xi+\zeta)-\theta(\xi)}  d\xi.
$$

\noindent \textbf{Step 1}. \emph{Intermediate bound for the Dirac part $L_{k,\delta}$}. For a fixed nonzero $\zeta\in \mathbb R^d$, we decompose any $\xi \in \mathbb R^d$ in the form $\xi= \xi^{\perp}+\tilde \xi$ where $\tilde \xi \in \mbox{Span}\{\zeta\}$ and $\xi^\perp \in \{\zeta\}^\perp$. We write $\tilde \xi=\sigma \frac{\zeta}{|\zeta|}$. This gives
$$
\mathcal F_{x,t}L_{k,\delta}(V)(\zeta,y,\omega)= \iint (g(\xi+\zeta)-g(\xi))e^{-i\xi y}  U_k(\zeta,\xi,\omega)\delta \left( \Omega (\zeta,\xi^\perp+\sigma \frac{\zeta}{|\zeta|},\omega)\right)  d\xi^\perp d\sigma
$$
where
$$
 \Omega (\zeta,\xi,\omega)=\omega+\theta(\xi+\zeta)-\theta(\xi).
$$
As $\nabla_\xi  \Omega =\nabla \theta(\xi+\zeta)-\nabla \theta(\xi)$ the uniform ellipticity assumption \eqref{id:uniform-ellipticity-theta} implies that
\begin{equation}\label{bd:Omega-derivative}
\frac{\zeta}{|\zeta|}.\nabla_{\xi} \Omega(\zeta,\xi,\omega) > \lambda_* |\zeta|
\end{equation}
for all $\zeta$, $\xi$ and $\omega$. Therefore, for each fixed $\zeta$, $\xi^\perp$ and $\omega$, $\Omega$ admits a unique zero of the form $\xi_0=\xi^\perp+\tilde \xi_0$ with $\tilde \xi_0=\sigma_0(\zeta,\xi^\perp,\omega)\frac{\zeta}{|\zeta|}$. Integrating along the $\sigma$ variables then gives
$$
\mathcal F_{x,t}L_{k,\delta}(V)(\zeta,y,\omega)= \int \frac{1}{\frac{\zeta}{|\zeta|}.\nabla_\xi \Omega(\zeta,\xi_0,\omega) }(g(\xi_0+\zeta)-g(\xi_0))e^{-i\xi_0 y} U_k(\zeta,\xi_0,\omega) d\xi^\perp.
$$
We introduce 
$$
g_\zeta(\xi)=\frac{g(\xi+\zeta)-g(\xi)}{|\zeta|}
$$
and
$$
\tilde U_{1,k} (\zeta,\xi , \omega)= \frac{|\zeta|}{ \frac{\zeta }{|\zeta|}.\nabla_\xi \Omega(\zeta,\xi,\omega) }g_\zeta(\xi)  U_k(\zeta,\xi,\omega) 
$$
and the above becomes
$$
\mathcal F_{x,t}L_{k,\delta}(V)(\zeta,y,\omega) = \int e^{-i\xi^\perp y-i\sigma_0\frac{\zeta}{|\zeta|}y}\tilde U_{1,k}(\zeta,\xi^\perp+\sigma_0\frac{\zeta}{|\zeta|},\omega)  d \xi^\perp .
$$
The integrand can be bounded using the one dimensional Sobolev embedding $W^{1,1}(\mathbb R)\rightarrow L^\infty(\mathbb R)$:
$$
|\tilde U_{1,k}(\zeta,\xi^\perp+\sigma_0\frac{\zeta}{|\zeta|},\omega)|\lesssim \| \tilde U_{1,k}(\zeta,\xi^\perp+\tilde \xi,\omega)\|_{L^1_{\tilde \xi}}+\| \nabla_\xi \tilde U_{1,k}(\zeta,\xi^\perp+\tilde \xi,\omega)\|_{L^1_{\tilde \xi}}.
$$
After integration along the remaining $\xi^\perp$ variable, this leads to the intermediate bound:
\begin{equation} \label{bd:intermediate-L4-delta}
|\mathcal F_{x,t}L_{k,\delta}(V)(\zeta,y,\omega)|  \lesssim  \| \tilde U_{1,k}(\zeta,\cdot,\omega)\|_{L^1_{ \xi}} +\| \nabla_\xi \tilde U_{1,k}(\zeta,\cdot,\omega)\|_{L^1_{\xi}}
\end{equation}
for any $\zeta$, $y$ and $\omega$. Note that $\tilde U_{1,3}$ depends on $\xi$ through $g$. 

\medskip

\noindent \textbf{Step 2}. \emph{Intermediate bound for the principal value part $L_{k,p.v.}$}. The reasoning is very similar to the previous one concerning $L_{k,\delta}$. For fixed nonzero $\zeta$ and $\omega$ we change again variables and write $\xi=\xi^\perp+\tilde \xi$ with $\tilde \xi= \sigma \frac{ \zeta}{|\zeta|} $ and $\xi^\perp \in \{\zeta\}^\perp$, so that
\begin{align*}
\mathcal F_{x,t}L_{k,p.v.}(V)(\zeta,y,\omega) &=  \iint g_\zeta(\xi)e^{-i\xi y} U_k(\zeta,\xi,\omega)\frac{|\zeta|}{\Omega(\zeta ,\xi^\perp+\sigma \frac{\zeta}{|\zeta|},\omega)}  d\xi^\perp d\sigma. 
\end{align*}
We introduce
$$
\tilde U_{2,k} (\zeta,\xi , \omega)= g_\zeta(\xi)  U_k(\zeta,\xi,\omega) 
$$
and the above becomes
\begin{align*}
\mathcal F_{x,t}L_{k,p.v.}(V)(\zeta,y,\omega) &=  \iint \tilde U_{2,k}(\zeta,\xi^\perp+\sigma\frac{\zeta}{|\zeta|},\omega)e^{-i\xi^\perp y-i\sigma y\frac{\zeta}{|\zeta|}} \frac{|\zeta|}{\Omega(\zeta,\xi^\perp+\sigma \frac{\zeta}{|\zeta|},\omega)}  d\xi^\perp d\sigma. 
\end{align*}
We recall that $\Omega$ satisfies \eqref{bd:Omega-derivative}. Furthermore, as $\nabla^{\otimes 3} \Tilde \theta\in L^{\infty}$, $\Omega$ satisfies by the mean value theorem that
\begin{align} \nonumber
\Big \lvert \nabla_\xi ( \frac{\zeta }{|\zeta|} \cdot \nabla_\xi \Omega(\zeta,\xi,\omega)) \Big \rvert & \lesssim \Big \lvert \nabla_\xi ( \frac{\zeta }{|\zeta|}\cdot \nabla_\xi \theta(\xi+\zeta))-\nabla_\xi ( \frac{\zeta }{|\zeta|}\cdot \nabla_\xi \theta(\xi)) \Big \rvert \\
 \label{bd:Omega-derivative2} &\leq C(\| \nabla^{\otimes 3}\Tilde \theta \|_{L^{\infty}},\lambda_*) |\zeta|.
\end{align}
One can thus integrate along the $\sigma$ variable and apply the estimate \eqref{bd:principal-value} with $\lambda=\lambda_*|\zeta|$ to bound
\begin{multline*}
 \left| \int \tilde U_{2,k}(\zeta,\xi^\perp+\sigma\frac{\zeta}{|\zeta|},\omega)e^{-i\xi^\perp y-i\sigma y\frac{\zeta}{|\zeta|}} \frac{|\zeta|}{\Omega(\zeta,\xi^\perp+\sigma \frac{\zeta}{|\zeta|},\omega)} d\sigma \right| \\
 \leq C(\| \nabla^{\otimes 3}\Tilde \theta \|_{L^{\infty}},\lambda_*) \| \tilde U_{2,k}(\zeta,\xi^\perp+\tilde \xi,\omega)\|_{L^2_{\tilde \xi}}+\| \nabla_\xi \tilde U_{2,k}(\zeta,\xi^\perp+\tilde \xi,\omega)\|_{L^{2}_{\tilde \xi}}
\end{multline*}
Integrating along the remaining $\xi^\perp$ variable, this leads to the intermediate bound
\begin{equation} \label{bd:intermediate-L4-pv}
|\mathcal F_{x,t}L_{k,p.v.}(V)(\zeta,y,\omega)|  \lesssim  \| \tilde U_{2,k}(\zeta,\cdot,\omega)\|_{L^1_{ \xi^\perp}L^2_{\tilde \xi}} +\| \nabla_\xi \tilde U_{2,k}(\zeta,\cdot,\omega)\|_{L^1_{\xi^\perp}L^2_{\tilde \xi}}
\end{equation}
for any $\zeta$, $y$ and $\omega$.

\medskip

\noindent \textbf{Step 3}. \emph{Bounds for $\tilde U_{1,k}$ and $\tilde U_{2,k}$}. By \eqref{bd:Omega-derivative} one has the equivalence
\[
|\tilde U_{1,k}(\zeta,\xi,\omega) |\approx |\tilde U_{2,k}(\zeta,\xi,\omega)|= |g_\zeta(\xi)  U_k(\zeta,\xi,\omega) |.
\]

We recall that $\omega$ denotes the dual Fourier variable of $t$. We have therefore by the H\"older inequalities
\begin{equation}\label{bd:tildeU1}
\| \tilde U_{1,k}(\zeta,\cdot,\cdot)\|_{L^1_{\xi}L^2_\omega}  \lesssim \| g_\zeta \|_{L^1_\xi} \|  U_k(\zeta,\cdot,\cdot) \|_{L^\infty_\xi L^2_\omega} 
\end{equation}
and similarly, for $q=1,2$,
\begin{align}
\nonumber 
\| \tilde U_{2,k}(\zeta,\cdot,\cdot)\|_{L^1_{\xi^\perp}L^q_{\tilde \xi}L^2_\omega} &  \lesssim \| g_\zeta \|_{L^1_{\xi^\perp}L^q_{\tilde \xi}} \| U_k(\zeta,\cdot,\cdot) \|_{L^\infty_\xi L^2_\omega}
\\
\label{bd:tildeU2} & \lesssim \| g_\zeta \|_{W^{1,1}_{\xi}}  \|  U_k(\zeta,\cdot,\cdot) \|_{L^\infty_\xi L^2_\omega} 
\end{align}
where we used the one-dimensional Sobolev embedding $W^{1,1}_{\tilde \xi}\rightarrow L^q_{\tilde \xi}$. Next, we differentiate
\[
\nabla_\xi \tilde U_{2,k}(\zeta,\xi , \omega)  = \nabla g_\zeta(\xi)  U_k(\zeta,\xi,\omega) + g_\zeta(\xi) \nabla_\xi  U_k(\zeta,\xi,\omega)  
\]
and estimate similarly that for $q=1,2$,
\begin{align}
\nonumber \| \nabla_\xi \tilde U_{2,k}(\zeta,\xi , \omega)  \|_{L^1_{\xi^\perp}L^q_{\tilde \xi}L^2_\omega} &  \lesssim \| \nabla g_\zeta \|_{L^1_{\xi^\perp}L^q_{\tilde \xi}} \| U_k(\zeta,\cdot,\cdot) \|_{L^\infty_\xi L^2_\omega}+ \| g_\zeta \|_{L^1_{\xi^\perp}L^q_{\tilde \xi}} \| \nabla_\xi U_k(\zeta,\cdot,\cdot) \|_{L^\infty_\xi L^2_\omega} \\
\label{bd:tildeU2-nablaxi}&\lesssim \| g_\zeta\|_{W^{2,1}} (\|  U_k(\zeta,\cdot,\cdot) \|_{L^\infty_\xi L^2_\omega}+\| \nabla_\xi U_k(\zeta,\cdot,\cdot) \|_{L^\infty_\xi L^2_\omega}).
\end{align}
Finally, we decompose
\begin{align*}
\left|\nabla_\xi \tilde U_{1,k}(\zeta,\xi , \omega)\right| & = \left|\frac{|\zeta|}{ \frac{\zeta }{|\zeta|}.\nabla_\xi \Omega(\zeta,\xi,\omega) } \nabla_\xi \tilde U_{2,k}+\nabla_\xi\left( \frac{|\zeta|}{ \frac{\zeta }{|\zeta|}.\nabla_\xi \Omega(\zeta,\xi,\omega) }\right) \tilde U_{2,k}\right| \\
&\lesssim C(\| \nabla^{\otimes 3}\Tilde \theta \|_{L^{\infty}},\lambda_*)(| \nabla_\xi \tilde U_{2,k}|+|\tilde U_{2,k}|)
\end{align*}
where we used \eqref{bd:Omega-derivative} and \eqref{bd:Omega-derivative2} to obtain that $|\frac{|\zeta|}{ \frac{\zeta }{|\zeta|}.\nabla_\xi \Omega(\zeta,\cdot,\cdot) }|+|\nabla_\xi (\frac{|\zeta|}{ \frac{\zeta }{|\zeta|}.\nabla_\xi \Omega(\zeta,\cdot,\cdot) })|\lesssim 1$. Using \eqref{bd:tildeU2} and \eqref{bd:tildeU2-nablaxi} with $q=1$ then shows
\begin{align} \label{bd:tildeU1-nablaxi}
\| \nabla_\xi \tilde U_{1,k}(\zeta,\xi , \omega)  \|_{L^1_{\xi}L^2_\omega} \lesssim C(\| \nabla^{\otimes 3}\Tilde \theta \|_{L^{\infty}},\lambda_*)\| g_\zeta\|_{W^{2,1}} (\|  U_k(\zeta,\cdot,\cdot) \|_{L^\infty_\xi L^2_\omega}+\| \nabla_\xi U_k(\zeta,\cdot,\cdot) \|_{L^\infty_\xi L^2_\omega}).
\end{align}

\noindent \textbf{Step 4}. \emph{Final bound}. Pick $y_0\in \mathbb R^d$. For any $\zeta \in \mathbb R^d$ and $\omega \in \mathbb R$, injecting the previous intermediate estimates \eqref{bd:intermediate-L4-delta} and \eqref{bd:intermediate-L4-pv} in \eqref{id:decomposition-L4} yields 
\begin{align*}
|\mathcal F_{x,t}L_{k}(V)(\zeta,y_0,\omega)|  \lesssim \sum_{j=0}^1 \| \nabla_\xi^j \tilde U_{1,k}(\zeta,\cdot,\omega)\|_{L^1_{ \xi}} + \| \nabla_\xi^j \tilde U_{2,k}(\zeta,\cdot,\omega)\|_{L^1_{ \xi^\perp}L^2_{\tilde \xi}} .
\end{align*}
Hence, for any fixed $\zeta \in \mathbb R^d$, applying the Parseval and Minkowski inequalities gives
\begin{align*}
\| \mathcal F_{x}L_{k}(V)(\zeta,y_0,\cdot) \|_{L^2_t} & = \| \mathcal F_{x,t}L_{k}(V)(\zeta,y_0,\cdot) \|_{L^2_\omega} \\
&\lesssim \sum_{j=0}^1 \| \nabla_\xi^j \tilde U_{1,k}(\zeta,\cdot,\omega)\|_{L^2_\omega L^1_{ \xi}} + \| \nabla_\xi^j \tilde U_{2,k}(\zeta,\cdot,\omega)\|_{L^2_\omega L^1_{ \xi^\perp}L^2_{\tilde \xi}} \\
&\lesssim \sum_{j=0}^1 \| \nabla_\xi^j \tilde U_{1,k}(\zeta,\cdot,\omega)\|_{ L^1_{ \xi} L^2_\omega } + \| \nabla_\xi^j \tilde U_{2,k}(\zeta,\cdot,\omega)\|_{L^1_{ \xi^\perp}L^2_{\tilde \xi}L^2_\omega } . 
\end{align*}
Using \eqref{bd:tildeU1}, \eqref{bd:tildeU2}, \eqref{bd:tildeU1-nablaxi} and \eqref{bd:tildeU2-nablaxi} to bound the right-hand side above we get
\begin{equation} \label{bd:L4-L2t}
\| \mathcal F_{x}L_{k}(V)(\zeta,y_0,\cdot) \|_{L^2_t} \lesssim C(\| \nabla^{\otimes 3}\Tilde \theta \|_{L^{\infty}},\lambda_*)\|  g_\zeta \|_{W^{2,1}_\xi} (\|  U_k(\zeta,\cdot,\cdot) \|_{L^\infty_\xi L^2_\omega}+\| \nabla_\xi U_k(\zeta,\cdot,\cdot) \|_{L^\infty_\xi L^2_\omega}).
\end{equation}
We now turn to bounding $g_\zeta $. We have for $j=0,1,2$ 
$$
\nabla^j g_\zeta(\xi)=\frac{\nabla^j g(\xi+\zeta)-\nabla^j g(\xi)}{|\zeta|} = \int_0^{1} \nabla^{j+1} g(\xi + t\zeta)\cdot \frac{\zeta}{|\zeta|}\ dt
$$
so that by Minkowski,
\begin{equation} \label{bd:g-zeta}
\| \nabla^j g_\zeta \|_{L^1_{\xi}}=\frac{1}{|\zeta|} \| \nabla^j g(\cdot+\zeta)- \nabla^j g(\cdot)\|_{L^1} \leq \| \nabla g\|_{\dot W^{j,1}}
\end{equation}
for any $\zeta \in \mathbb R^d$. Injecting \eqref{bd:g-zeta} in \eqref{bd:L4-L2t}, one obtains
$$
\| \mathcal F_x L_{k} (V)(\zeta,y_0,\cdot)\|_{L^2_{t}}\lesssim C(\| \nabla^{\otimes 3}\Tilde \theta \|_{L^{\infty}},\lambda_*) \| \nabla g \|_{W^{2,1}} (\|  U_k(\zeta,\cdot,\cdot) \|_{L^\infty_\xi L^2_\omega}+\| \nabla_\xi U_k(\zeta,\cdot,\cdot) \|_{L^\infty_\xi L^2_\omega}).
$$

For $k=4$, we note that
\begin{align*}
U_4(\zeta,\xi, \omega) & = \mathcal F_{x,y,t}(w(y) V(x,y,t))(\zeta,-\xi,\omega), \\
\nabla_\xi U_4(\zeta,\xi, \omega) & = \mathcal F_{x,y,t}(iyw(y) V(x,y,t))(\zeta,-\xi,\omega).
\end{align*}
We deduce 
\[
\|U_4\|_{L^2_\zeta L^\infty_\xi L^2_\omega} \leq  \|\mathcal F_{t,x}V(\zeta,y,\omega)\|_{L^2_\zeta L^1_{|w(y)|dy} L^2_\omega}
\]
and thus by Minkowski and Parseval inequality
\[
\|U_4\|_{L^2_\zeta L^\infty_\xi L^2_\omega} \leq  \|V\|_{ L^1_{|w(y)|dy} L^2_{t,x}}.
\]
For similar reasons
\[
\|\nabla_\xi U_4\|_{L^2_\zeta L^\infty_\xi L^2_\omega} \leq  \|V\|_{ L^1_{|y w(y)|dy} L^2_{t,x}}.
\]
It remains to use that $\an y w$ is a finite measure to get
\[
\|\nabla_\xi U_4\|_{L^2_\zeta L^\infty_\xi L^2_\omega} + \| U_4\|_{L^2_\zeta L^\infty_\xi L^2_\omega}\lesssim_w \|V\|_{L^\infty_y L^2_{t,x}}.
\]

For $k=3$, we note that
\[
U_3(\zeta,\xi, \omega) = \hat w(\zeta) \mathcal F_{x,t}( V(x,0,t))(\zeta,\omega), \quad \nabla_\xi U_3(\zeta,\xi, \omega) = 0.
\]
We deduce 
\[
\|U_3\|_{L^2_\zeta L^\infty_\xi L^2_\omega} \leq  \|\hat w\|_{L^\infty} \|V(y=0)\|_{L^2_{t,x}}
\]
which suffices to conclude.

\medskip

\noindent \textbf{Step 5}. \emph{Preservation of continuity}. Take $y_0,y \in \R^d$. We have 
\begin{multline*}
\mathcal F_{t,x} L_{k,\delta}(V)(\zeta,y_0 + y, \omega)- \mathcal F_{t,x} L_{k,\delta}(V)(\zeta,y_0, \omega)\\
= \int \Big(e^{-i \xi^\perp y- i\sigma_0 \frac{\zeta}{|\zeta|}y} -1\Big)e^{-i \xi^\perp y_0- i\sigma_0 \frac{\zeta}{|\zeta|}y_0} \tilde U_{1,k}(\zeta,\xi^\perp+\sigma_0 \frac{\zeta}{|\zeta|},\omega) d\xi^\perp.
\end{multline*}
We deduce that
\begin{multline*}
\|L_{k,\delta}(V)(\cdot,y_0+y,\cdot)-L_{k,\delta}(V)(\cdot,y_0,\cdot)\|_{L^2_{t,x}} \\
\lesssim \|(e^{i\xi y}-1)\tilde U_{1,k}\|_{L^2_{\zeta,\omega},L^1_\xi}+\|(e^{i\xi y}-1)\nabla _\xi \tilde U_{1,k}\|_{L^2_{\zeta,\omega},L^1_\xi} + |y|\|\tilde U_{1,k}\|_{L^2_{\zeta,\omega},L^1_\xi}.
\end{multline*}
Since $U_{1,k}$ and $\nabla_\xi U_{1,k}$ belong to $L^2_{\zeta,\omega},L^1_\xi$ we get by the dominated convergence Theorem that
\[
\|L_{k,\delta}(V)(\cdot,y_0+y,\cdot)-L_{k,\delta}(V)(\cdot,y_0,\cdot)\|_{L^2_{t,x}} \rightarrow 0
\]
as $y$ goes to $0$.

A similar reasoning yields the continuity of $L_{k,p.v.}$ in $y$. 

\end{proof}

\begin{proposition}\label{prop:invl3-l4} The linear operators $L_3$ and $L_4$ are continuous from $E_V$ to $E_V$. Moreover, under the smallness assumption
$$
\| \la y \ra w \|_{L^1} \| \nabla g \|_{W^{2,1}}\leq C(\| \nabla^{\otimes 3}\Tilde \theta \|_{L^{\infty}},\lambda_*),
$$
the linear operator $1-L_3-L_4$ is invertible on $E_V$.
\end{proposition}

\begin{proof} Because $L_3$ and $L_4$ are bounded Fourier multipliers in $x$ they commute with the laplacian and thus $L_3$ and $L_4$ are continuous from $\mathcal C_y L^2_t \dot H^{-1/2}_x\cap L^\infty_y L^2_t \dot H^{-1/2}_x$ to itself and from $\mathcal C_y L^2_t \dot H^{s_c}_x\cap L^\infty_y L^2_t \dot H^{s_c}_x$ to itself. We deduce that $L_3$ and $L_4$ are continuous from $\mathcal C_y L^2_t B_2^{-1/2, s_x}\cap L^\infty_y L^2_t B_2^{-1/2,s_c}$ to itself. Besides their operator norms go to 0 as the $W^{2,1}$ norm of $\nabla g$ goes to $0$ which ensures the invertibility of $1-L_3-L_4$ on $\mathcal C_y L^2_t B_2^{-1/2, s_x}\cap L^\infty_y L^2_t B_2^{-1/2,s_c}$ as long as $g$ is small enough. From this, we also deduce that $L_3$ and $L_4$ are continuous from $E_V$ to $L^\infty_y L^2_{t,x}$.

From Corollaries \ref{corollary:l1} \ref{corollary:l2}, we know that $L_1$ and $L_2$ are continuous from $\mathcal C_y L^2_t B_2^{-1/2, s_x}\cap L^\infty_y L^2_t B_2^{-1/2,s_c}$ to $E_Z$, we get in particular that they are continuous from $\mathcal C_y L^2_t B_2^{-1/2, s_x}\cap L^\infty_y L^2_t B_2^{-1/2,s_c}$ to $L^{d+2}_{t,x} L^2_\omega$. Because for $k=1,2$,
\begin{align*}
|L_{k+2} (V) (x,y,t)| & = |\E(\bar Y(x+y,t) L_1(V)(x,t)) + \E(\overline{L_1(V)} (x+y,t) Y(x,t))| \\
& \leq \Big(\int g\Big) (\|L_k(V)(x,t)\|_{L^2_\omega} + \|L_k(V)(x+y,t)\|_{L^2_\omega})
\end{align*}
we deduce that
\[
\|L_{k+2}(V)\|_{L^{d+2}_{t,x}} \lesssim_g \|L_k(V)\|_{E_Z} \lesssim_g \|V\|_{ L^\infty_y L^2_t B_2^{-1/2,s_c}}.
\]
Because $\frac{d+2}{2} \in [2,d+2]$, we deduce that for $k=1,2$,
\[
\|L_{k+2}(V)\|_{L^{(d+2)/2}_{t,x}}  \lesssim_g \|V\|_{ L^\infty_y L^2_t B_2^{-1/2,s_c}}.
\]

This ensures that $L_3$ and $L_4$ are continuous from $\mathcal C_y L^2_t B_2^{-1/2, s_x}\cap L^\infty_y L^2_t B_2^{-1/2,s_c}$ to $E_V$. Finally, noticing that 
\[
(1-L_3-L_4)^{-1} = 1+(L_3+L_4)(1-L_3-L_4)^{-1}
\]
We get that the restriction of $(1-L_3-L_4)^{-1}$ to $E_V$ is continuous from $E_V$ to $E_V$ which makes $1-L_3-L_4$ invertible on $E_V$.
\end{proof}

\section{Bilinear estimates}\label{sec:bilinear}

\subsection{Quadratic terms in the perturbation}

\begin{proposition}\label{prop:estQ12} There exists $C$ such that for all $Z,V \in E_Z \times E_V$, and $k=1,2$, we have 
\[
\|Q_k(Z,V)\|_{Z} \leq C \| \la y\ra w\|_{M} \|Z\|_Z \|V\|_V.
\]
\end{proposition}

\begin{proof} Set 
\[
\tilde V_1 = w* V(y=0) Z,\quad \tilde V_2 (x) = - \int dz w(z) V(x,z) Z(x+z).
\]

We have 
\[
Q_k(Z,V) = -i\int_0^t S(t-\tau) \tilde V_k(\tau) d\tau .
\]
By Strichartz estimates (\ref{ineq:strichartz-derivatives}), we have 
\[
\|Q_k(Z,V)\|_Z \lesssim \|\tilde V_k\|_{L^{p'}(\R,W^{s_c,p'}(\R^d,L^2(\Omega)))}.
\]
We use the generalized Leibniz rule with
\[
\frac1{p'} = \frac1{p} + \frac2{d+2},\quad \frac1{p'} = \frac1{d+2} +  \frac12
\]
and Hölder inequality in time to get
\begin{multline*}
\|\tilde V_1\|_{L^{p'}(\R,W^{s_c,p'}(\R^d,L^2(\Omega)))} \\
\lesssim \|w* V(y=0)\|_{L^{(d+2)/2}(\R\times \R^d) \cap L^2(\R,H^{s_c}(\R^d))} \|Z\|_{L^p(\R,W^{s_c,p}(\R^d,L^2(\Omega)))\cap L^{d+2}(\R\times \R^d,L^2(\Omega))} .
\end{multline*}
Using that $w$ is a finite measure, we get
\[
\|\tilde V_1\|_{L^{p'}(\R,W^{s_c,p'}(\R^d,L^2(\Omega)))} \lesssim \|V\|_{V} \|Z\|_{Z} .
\]

For $\tilde V_2$, we use the Minkowski's inequality to get
\[
\|\tilde V_2\|_{L^{p'}(\R,W^{s_c,p'}(\R^d,L^2(\Omega)))} \leq \int |w(z)| \| V(z) T_z Z\|_{L^{p'}(\R,W^{s_c,p'}(\R^d,L^2(\Omega)))}
\]
We use again the generalised Leibniz rule to get
\[
\|\tilde V_2\|_{L^{p'}(\R,W^{s_c,p'}(\R^d,L^2(\Omega)))} \lesssim \int |w(z)| \| V(z)\|_{L^{(d+2)/2}(\R\times \R^d) \cap L^2(\R,H^{s_c}(\R^d))} \|T_z Z\|_{Z}
\]
The $Z$ norm being invariant under the action of translations, and because of the definition of $V$, we get
\[
\|\tilde V_2\|_{L^{p'}(\R,W^{s_c,p'}(\R^d,L^2(\Omega)))} \lesssim \int |w(z)| \| V\|_{V} \| Z\|_{Z}.
\]
We use that $\an z w$ is a finite measure to conclude.
\end{proof}

\subsection{Quadratic terms in the correlation function}

\begin{proposition}\label{prop:embed} The bilinear map
\[
\begin{array}{l l l}
     E_Z \times E_Z &  \mapsto & E_V  \\
     (Z,Z')& \mapsto & ((x,y)\mapsto \E(Z(x+y)Z'(x))
\end{array}
\]
is well-defined and continuous.
\end{proposition}

\begin{proof} 

\underline{Estimate in $C_y L^{(d+2)/2}_{t,x}$}. We recall that $T_y$ is the translation such that $T_y u(x) = u(x+y)$. For a given $y\in \R^d$, we have 
\[
\|\E(T_y Z Z')\|_{L^{(d+2)/2}(\R\times \R^d)} \leq \|T_y Z\|_{L^{d+2}(\R\times \R^d,L^2(\Omega))} \| Z'\|_{L^{d+2}(\R\times \R^d,L^2(\Omega))}.
\]
We recall that Lebesgue norms are invariant under the action of translation, and that for any $u\in L^{d+2}_{t,x}L^2_\omega$ the curve $y\mapsto T_y u$ is continuous in $L^{d+2}_{t,x}L^2_\omega$. Thus, the above estimate implies
\[
\|\E(T_y Z Z')\|_{L^{(d+2)/2}(\R\times \R^d)} \leq \| Z\|_{Z} \| Z'\|_{Z}
\]
as well as $\E(T_y Z Z')\in C_yL^{(d+2)/2}_{t,x}$. We deduce that
\[
\|\E(T_y Z Z')\|_{L^\infty(\R^d, L^{(d+2)/2}(\R^d\times \R))} \leq \| Z\|_{Z} \| Z'\|_{Z}.
\]

\noindent \underline{Estimate in $C_y L^2_t H^{s_c}_x$}. We use the generalised Leibniz rule, with
\[
\frac12 = \frac{d}{2(d+2)} + \frac1{d+2} = \frac1{p} + \frac1{d+2}
\]
to get
\begin{align*}
\|\E(T_y Z Z')\|_{L^2(\R,H^{s_c}(\R^d))} \leq (\|T_y Z\|_{L^p(\R,W^{s_c,p}(\R^d,L^2(\Omega)))} + \|T_y Z\|_{L^{d+2}(\R\times \R^d,L^2(\Omega))} )\\
\times (\| Z'\|_{L^p(\R,W^{s_c,p}(\R^d,L^2(\Omega)))} + \| Z'\|_{L^{d+2}(\R\times \R^d,L^2(\Omega))} ) .
\end{align*}
Again, we recall the invariance of Sobolev norms under translations, and that for any $u\in L^{p}_{t}W^{s_c,p}_xL^2_\omega$ the curve $y\mapsto T_y u$ is continuous in $L^{p}_{t}W^{s_c,p}_xL^2_\omega$. This shows
\[
\sup_{y\in \mathbb R^d} \, \|\E(T_y Z Z')\|_{L^2(\R,H^{s_c}(\R^d))} \leq \| Z\|_{Z} \| Z'\|_{Z}
\]
as well as $\E(T_y Z Z')\in C_y L^2_t H^{s_c}_x L^2_\omega$.

\noindent \underline{Estimate in $C_y L^2_t B^{-1/2,s_c}_2$}. For this part of the $V$ norm, we use that
\[
\|\cdot\|_{L^2(\R, B_2^{-1/2,s_c}(\R^d))} \leq \|\cdot\|_{L^2(\R,\dot H^{-1/2}(\R^d))}  + \|\cdot\|_{L^2(\R,H^{s_c}(\R^d))}.
\]

The second term in the right-hand side has already been estimated. For the 
\[
L^2(\R,\dot H^{-1/2}(\R^d))
\]
norm, we use homogeneous Sobolev estimates to get
\begin{align*}
\|\E(T_y Z Z') \|_{L^2(\R,\dot H^{-1/2}(\R^d))} & \lesssim \|\E(T_y Z Z') \|_{L^2(\R, L^{q/2}(\R^d))} \\
& \lesssim \|T_y Z\|_{L^4(\R,L^q(\R^d,L^2(\Omega)))} \| Z'\|_{L^4(\R,L^q(\R^d,L^2(\Omega)))}.
\end{align*}
Again, we recall the invariance of Lebesgue norms under the action of translations, and that for any $u\in L^4_{t}L^q_x L^2_\omega$ the curve $y\mapsto T_y u$ is continuous in $L^4_{t}L^q_x L^2_\omega$. This implies
\[
\sup_{y\in \mathbb R^d} \, 
\|\E(T_y Z Z') \|_{L^2(\R,B_2^{-1/2,0}(\R^d))} \lesssim \| Z\|_{Z} \| Z'\|_{Z}.
\]
as well as $\E(T_y Z Z')\in C_y L^2_t B_2^{-1/2,0}L^2_\omega $.

\end{proof}

\begin{proposition}\label{prop:estQ34} There exists $C$ such that for all $Z,V \in E_Z \times E_V$ and $k=3,4$, we have 
\[
\|Q_k(Z,V)\|_V \leq C \| \la y\ra w\|_{M} \|Z\|_Z \|V\|_V.
\]
\end{proposition}

\begin{proof} As in the proof of \Cref{prop:estQ12}, we set
\[
\tilde V_1 = w* V(y=0) Z,\quad \tilde V_2 (x) = - \int dz w(z) V(x,z) Z(x+z).
\]
We recall that $Q_3$ and $Q_4$ are defined as for $k=3,4$
\[
Q_k (Z,V)(y) = \E(T_y \bar Y Q_{k-2}) + \E( T_y \bar Q_{k-2} Y).
\]

\noindent\underline{Estimate in $\mathcal{C}_yL_{t,x}^{\frac{d+2}{2}}$.} Using that $Q_{k-2}(Z,V)$ is continuous from $E_Z\times E_V$ to $L^{d+2}(\R\times \R^d )$, by \Cref{prop:estQ12}, and since $Y\in L^\infty(\R\times \R^d, L^2(\Omega))$ and Lebesgue norms are invariant under the action of translations, we get that 
\[
\|Q_k(Z,V)(y)\|_{L^{d+2}(\R\times \R^d)} \lesssim \|Z\|_Z\|V\|_V
\]
Moreover, again by \Cref{prop:estQ12}, the same inequality holds with the $L^{\frac{2(d+2)}d}$ norm. Then, by interpolation we get $Q_k(Z,V)\in L_y^\infty,L_{t,x}^{(d+2)/2}.$ \\
To get the continuity with respect to $y$ we use that for any $u\in L_{t,x}^{(d+2)/2}$ the curve $y\mapsto T_yu$ is continuous in $ L_{t,x}^{(d+2)/2}$. 

\noindent \underline{Estimate in $C_y L^2_t B^{-1/2,s_c}_x$}. We prove that 
\[
\|Q_k(Z,V)(y)\|_{L^2(\R,B_2^{-1/2,s_c}(\R^d))} \lesssim \|Z\|_Z \|V\|_V.
\]

We remark that because of the invariance of Besov norms under conjugation and translations, we have
\[
\|\E( T_y \bar Q_{k-2} Y)\|_{L^2(\R,B_2^{-1/2,s_c}(\R^d))} = \|\E(  Q_{k-2} T_{-y}\bar Y)\|_{L^2(\R,B_2^{-1/2,s_c}(\R^d))}.
\]
Therefore, we bound only uniformly in $y$,
\[
\|\E( T_y \bar Q_{k-2} Y)\|_{L^2(\R,B_2^{-1/2,s_c}(\R^d))}.
\]
We proceed by duality and treat separately low and high frequencies. Take 
\[
U \in L^2(\R, B_2^{1/2}(\R^d))
\]
such that $U$ is localised in low frequencies. We have 
\[
\an{U, \E(  T_{y}Q_{k-2} \bar Y)} = \E(\an{U Y,T_y Q_{k-2}}).
\]
Using the definition of $Q_{k-2}$ and because the linear flow $S$ commutes with translations we get
\[
\an{U, \E( T_{y} Q_{k-2} \bar Y)} = \E(\an{U Y, -i\int_{0}^t S(t-\tau)T_{y} \tilde V_{k-2}(\tau)}) = \E(\an{\int_{\tau}^\infty U(t) Y(t) dt,T_{y} \tilde V_{k-2}}). 
\]
Set $p_1 = 2\frac{d+2}{d-2}$, we have 
\[
|\an{U, \E(  T_y Q_{k-2}\bar Y)}| \leq \|\int_{\tau}^\infty U(t) Y(t) dt\|_{L^{p_1}(\R\times \R^d, L^2(\Omega))}\| T_y\tilde V_{k-2}(\tau)\|_{L^{p_1'}(\R\times \R^d,L^2(\Omega))}. 
\]
We have that $\frac2{p_1} + \frac{d}{p_1} = \frac{d}{2} - 1$ and thus, by Proposition \ref{prop:magicProposition},
\[
\|\int_{\tau}^\infty U(t) Y(t) dt\|_{L^{p_1}(\R\times \R^d, L^2(\Omega))} \lesssim \|U\|_{L^2(\R, B_2^{1/2,1}(\R^d))}
\]
and given that $U$ is localised in low frequencies, we have 
\[
\|\int_{\tau}^\infty U(t) Y(t) dt\|_{L^{p_1}(\R\times \R^d, L^2(\Omega))} \lesssim \| U\|_{L^2(\R, B_2^{1/2,0}(\R^d))}.
\]
We remark that $\frac1{p_1'} = \frac12 + \frac2{d+2}$, and thus by H\"older inequality, we get
\[
\|\tilde V_1\|_{L^{p_1'}(\R\times \R^d, L^2(\Omega))} \leq \|w*V(y=0)\|_{L^2(\R\times \R^d)} \|Z\|_{L^{(d+2)/2}(\R\times \R^d,L^2(\Omega))}.
\]
By definition of the $V$ norm, we have 
\[
\|w*V(y=0)\|_{L^2(\R\times \R^d)} \lesssim \|V\|_V
\]
and because in dimension higher than $4$, $\frac{d+2}{2} \in [2\frac{d+2}{d}, d+2]$, we have 
\[
\|Z\|_{L^{(d+2)/2}(\R\times \R^d,L^2(\Omega))}\leq \|Z\|_Z.
\]
For the same reasons we have 
\[
\|\tilde V_2\|_{L^{p_1'}(\R\times \R^d, L^2(\Omega))} \leq \int dz |w(z)| \|V(z)\|_{L^2(\R\times \R^d)} \|Z\|_Z
\]
and we use the definition of the $V$ norm and the fact that $\an z w$ is a finite measure to conclude the analysis in low frequencies.\\
We turn to high frequencies, we take $U\in L^2(\R\times \R^d)$ localised in high frequencies.\\ 
We take $n\in \N$ and consider
\[
\an{U, \nabla^{\otimes n} \E( T_{y} Q_{k-2} \bar Y)} = \sum_{j=0}^n \parmi{j}{n} \E(\an{U,\nabla^{\otimes (n-j)}  \bar Y \otimes \nabla^{\otimes j} T_y Q_{k-2} }.
\]
With the same computation as previously, we have 
\[
\an{U, \nabla^{\otimes n} \E( T_{y} Q_{k-2} \bar Y)} = \sum_{j=0}^n \parmi{j}{n} \E(\an{\int_{\tau}^\infty S(\tau-t) U(t)\nabla^{\otimes (n-j)}  Y(t)dt,  \nabla^{\otimes j} T_y \tilde V_{k-2} }.
\]
By H\"older's inequality,
\begin{multline*}
|\an{U, \nabla^{\otimes n} \E( T_{y} Q_{k-2} \bar Y)}|\\
\leq  \sum_{j=0}^n \parmi{j}{n} \|\int_{\tau}^\infty S(\tau-t) U(t)\nabla^{\otimes (n-j)}  Y(t)dt\|_{L^p(\R\times\R^d,L^2(\Omega)}  \|\nabla^{\otimes j} T_y\tilde V_{k-2}\|_{L^{p'}(\R\times \R^d,L^2(\Omega))}.
\end{multline*}
By Proposition \ref{prop:magicProposition}, applied to $\nabla^{\otimes (n-j)}  Y(t)$ we have 
\[
 \|\int_{\tau}^\infty S(\tau-t) U(t)\nabla^{\otimes (n-j)}  Y(t)dt\|_{L^p(\R\times\R^d,L^2(\Omega)} \lesssim \|U\|_{L^2(\R,B_2^{-1/2,0}(\R^d))}.
\]
We deduce that 
\[
\tilde V_{k-2} \mapsto \Pi \E(Y \int_{0}^t S(t-\tau) \tilde V_{k-2}(\tau) d\tau
\]
where $\Pi$ projects into high frequencies is continuous from $L^{p'}(\R,W^{n,p'}( \R^d,L^2(\Omega)))$ to $L^\infty_y,L^2(\R, H^n)$. By interpolation, we get that it is continuous from \\ $L^{p'}(\R,W^{s_c,p'}( \R^d,L^2(\Omega)))$ to $L^\infty_y,L^2(\R, H^{s_c})$. 
\\We refer to the proof of Proposition \ref{prop:estQ12} to get that
\[
\|\tilde V_{k-2}\|_{L^{p'}(\R,W^{s_c,p'}( \R^d,L^2(\Omega)))}\lesssim \|Z\|_Z \|V\|_V.
\]
Moreover, for any $u\in L^{p_1'}(\R\times \R^d,L^2(\Omega))$ and for any $v\in L^{p'}(\R,W^{s_c,p'}( \R^d,L^2(\Omega)))$, the curves $y\mapsto T_y u$ and $y\mapsto T_y v$ are continuous respectively in $L^{p_1'}(\R\times \R^d,L^2(\Omega))$ and in $L^{p'}(\R,W^{s_c,p'}( \R^d,L^2(\Omega)))$. Thus we get the continuity with respect to $y$.
\end{proof}

\section{Proof of the theorem} \label{sec:fixed-point}

\subsection{Free evolution of the initial data}

\begin{proposition}\label{prop:ID-Z}

Under the assumptions of Proposition \ref{pr:Strichartz-free-evolution}, for any $Z_0\in L^2_\omega H^{s_c}_x$ one has $S(t)Z_0\in E_Z$ with
\[
\| S(t)Z_0\|_Z\lesssim \|Z_0\|_{L^2_\omega H^{s_c}_x}.
\]

\end{proposition}

\begin{proof}

The proof is exactly the same as that of Proposition 4.2 in \cite{CdS}, since the group $S(t)$ enjoys the same Strichartz estimates as the usual Schr\"odinger group $e^{it\Delta}$ by Proposition \ref{pr:Strichartz-free-evolution}.

\end{proof}

\begin{proposition} \label{pr:free-evolution-Z0}

Under the assumptions of Proposition \ref{pr:Strichartz-free-evolution}, if $\la \xi \ra^{2\lfloor s_c\rfloor}g\in W^{2,1} $ and $\Tilde \theta \in W^{4,1}$, then for any $Z_0\in L^2_\omega H^{s_c}_x$ one has
$$
\| \E [\overline{Y(x+y)}S(t)Z_0(x)+\overline{S(t)Z_0(x+y)}Y(x)]\|_{L^\infty_y L^{\frac{d+2}{2}}_{t,x}}\lesssim \| Z_0\|_{L^2_\omega H^{s_c}_x}
$$
and moreover if $Z_0\in L_x^\frac{2d}{d+2} L_\omega^2 $ then $\E [\overline{Y(x+y)}S(t)Z_0(x)+\overline{S(t)Z_0(x+y)}Y(x)]\in C_yL_t^{2}B_2^{-\frac{1}{2},s_c}$, with: $$
\| \mathbb E[\overline{Y(x)}S(t)T_y Z_0(x)] \|_{L^\infty_y L^2_t  B_2^{-1/2,s_c}}\lesssim \| Z_0\|_{L^{2d/(d+2)}_x L^2_\omega }+\| Z_0\|_{H^{s_c}_x L^2_\omega }.
$$

\end{proposition}

\begin{proof}

We recall that $T_y$ is the translation $T_y u(x)=u(x+y)$. Noticing that the second term equals the first one up to complex conjugation and changing variables $(x,y)\mapsto (x+y,-y)$, it suffices to bound
$$
\mathbb E[\overline{Y(x)}S(t)T_y Z_0(x)].
$$

\medskip

\noindent \underline{The $C_y L^{(d+2)/2}_{t,x}$ estimate}. By Proposition \ref{prop:ID-Z}, we have
$$
\| S(t)Z_0\|_{L^{p}_{t,x}L^2_\omega}+\| S(t)Z_0\|_{L^{d+2}_{t,x}L^2_\omega}\lesssim \| Z_0\|_{L^2_\omega H^{s_c}_x}.
$$
As $p\leq \frac{d+2}{2}\leq d+2$ this implies
$$
\| S(t)Z_0\|_{L^{\frac{d+2}{2}}_{t,x}L^2_\omega}\lesssim \| Z_0\|_{L^2_\omega H^{s_c}_x}
$$
by interpolation. From the formula \eqref{id:correlation-function-Yf} giving the correlation function of $Y$, we deduce $Y\in L^\infty_{t,x}L^2_\omega$. These estimates imply, via the Cauchy-Schwarz and H\"older inequalities,
$$
\| \mathbb E[\overline{Y(x)}S(t)T_y Z_0(x)] \|_{L^{\frac{d+2}{2}}_{t,x}}\lesssim \| T_y Z_0\|_{L^2_\omega H^{s_c}}.
$$
We recall that Lebesgue norms are invariant under translations, and that for any $u\in L^2_\omega H^{s_c}$, the curve $y\mapsto T_y u$ is continuous in $L^2_\omega H^{s_c}_x$. Hence $\mathbb E[\overline{Y(x)}S(t)T_y Z_0(x)]\in C_y L^{(d+2)/2}_{t,x}$ with
$$
\| \mathbb E[\overline{Y(x)}S(t)T_y Z_0(x)] \|_{L^\infty_y L^{\frac{d+2}{2}}_{t,x}}\lesssim \| T_y Z_0\|_{L^2_\omega H^{s_c}}.
$$

\noindent \underline{The $C_y L^2_t \Dot{H}_x^{-1/2} $ estimate}. We reason by duality. For $U\in L^2_t \Dot{H}^{1/2}_x$ with $\| U\|_{L^2_t \Dot{H}_x^{1/2}}=1$ we will prove
\begin{equation} \label{bd:correlation-function-for-free-evolution-inter}
\int_0^\infty  \mathbb E[\overline{Y(x)}S(t)T_y Z_0(x)] \bar U dxdt \lesssim \| T_y Z_0\|_{L^{2d/(d+2)}_x L^2_\omega }
\end{equation}
which will imply
$$
\| \mathbb E[\overline{Y(x)}S(t)T_y Z_0(x)]  \|_{L^2_t\Dot{H}^{-1/2}_x} \lesssim \| T_y Z_0\|_{L^{2d/(d+2)}_x L^2_\omega }.
$$
The above estimate implies $\mathbb E[\overline{Y(x)}S(t)T_y Z_0(x)]\in C_y L^2_t\Dot{H}_x^{-1/2}$ with 
\begin{equation} \label{bd:correlation-from-free-evolution--1/2}
\| \mathbb E[\overline{Y(x)}S(t)T_y Z_0(x)]  \|_{L^\infty_y L^2_t\Dot{H}^{-1/2}_x} \lesssim \| Z_0\|_{L^{2d/(d+2)}_x L^2_\omega }.
\end{equation}
We are thus left to proving \eqref{bd:correlation-function-for-free-evolution-inter}. Using Fubini and $S(t)^*=S(-t)$ we have
$$
\int_0^\infty\mathbb E[\overline{Y(x)}S(t)T_y Z_0(x)] \bar U dxdt= \int \mathbb E \left[ T_y Z_0 \int_0^\infty \overline{S(-t)YU}dt  \right] dx.
$$
By Proposition \ref{prop:magicProposition} with $\sigma_1=1$, $\sigma=0$, $p_1=\infty$ and $q_1=\frac{2d}{d-2}$ we have $ \int_0^\infty S(-t)YU dt \in L^{2d/(d-2)}_xL^2_\omega$ with $\| \int_0^\infty S(-t)YU \|_{ L^{2d/(d-2)}_xL^2_\omega}\lesssim 1$. This estimate implies \eqref{bd:correlation-function-for-free-evolution-inter} by H\"older's inequality.

\medskip

\noindent \underline{The $\mathcal C_y L^2_t H_x^{s_c} $ estimate}. We notice that by the Leibniz rule, in order to estimate 
\[
\nabla_x^{\lfloor s_c\rfloor}(YS(t)T_yZ_0),
\]
it is sufficient to estimate $\partial^\alpha YS(t)T_y \partial^\beta Z_0$ for all multi-indices $\alpha,\beta \in \mathbb N^d$ with $\sum_i \alpha_i+\sum_i \beta_i=\lfloor s_c \rfloor$. We pick such $\alpha$ and $\beta$, and, for again a duality argument, take $U\in H^{\lfloor s_c\rfloor-s_c}_x L^2_\omega$ with $\|U\|_{H^{\lfloor s_c\rfloor-s_c}_x L^2_\omega}=1$. We assume $U$ is supported away from the origin in frequencies, i.e. $\hat U(\xi)=0$ for $|\xi|\leq 1$. We will show
\begin{equation} \label{bd:correlation-function-for-free-evolution-inter-2}
\int_0^\infty  \mathbb E[\overline{\partial^\alpha Y(x)}S(t)T_y \partial^\beta Z_0(x)] \bar U dxdt \lesssim \| T_y Z_0\|_{H^{s_c}_x L^2_\omega }.
\end{equation}
This estimate, via duality, and combined with \eqref{bd:correlation-from-free-evolution--1/2} for the low frequencies, will show
$$
\| \mathbb E[\overline{Y(x)}S(t)T_y Z_0(x)]  \|_{L^2_tH^{s_c}_x} \lesssim \| T_y Z_0\|_{L^{2d/(d+2)}_x L^2_\omega }+\| T_y Z_0\|_{H^{s_c}_x L^2_\omega }.
$$
The above estimate implies $\mathbb E[\overline{Y(x)}S(t)T_y Z_0(x)]\in C_y L^2_tH_x^{s_c}$ with 
\begin{equation} \label{bd:correlation-from-free-evolution-s_c}
\| \mathbb E[\overline{Y(x)}S(t)T_y Z_0(x)]  \|_{L^\infty_y L^2_tH^{s_c}_x} \lesssim \| Z_0\|_{L^{2d/(d+2)}_x L^2_\omega }+\| T_y Z_0\|_{H^{s_c}_x L^2_\omega }.
\end{equation}
It remains to show \eqref{bd:correlation-function-for-free-evolution-inter-2}. We have
\begin{equation} \label{bd:correlation-function-for-free-evolution-inter-3}
\int_0^\infty \mathbb E[\overline{\partial^\alpha Y(x)}S(t)T_y \partial^\beta Z_0(x)] \bar U dxdt= \int \mathbb E \left[ T_y \partial^\beta Z_0 \int_0^\infty \overline{S(-t)\partial^\alpha YU}dt  \right] dx.
\end{equation}
Note that
$$
\partial^\alpha Y(t,x)=\int f_\alpha (\xi)e^{i(\xi x-\theta(\xi)t)}dW(\xi)
$$
where $f_\alpha(\xi)= (i\xi_1)^{\alpha_1} ... (i\xi_d)^{\alpha_d} f(\xi)$. We apply Proposition \ref{prop:magicProposition} with momenta distribution function $f_\alpha$, $\sigma_1=0$, $\sigma=0$, $p_1=\infty$ and $q_1=2$ and obtain
\begin{equation} \label{bd:correlation-function-for-free-evolution-inter-4}
\| \int_0^\infty S(-t)\partial^\alpha YU \|_{L^2_{x,\omega}}\lesssim \| U\|_{L^2_t H^{-1/2}_x}\lesssim \| U\|_{L^2_t H^{\lfloor s_c\rfloor-s_c}_x}=1
\end{equation}
where for the before last inequality we used $\lfloor s_c\rfloor -s_c\in \{-1/2,0\}$ and that $U$ is located away from the origin in frequencies. Using the Cauchy-Schwarz inequality and \eqref{bd:correlation-function-for-free-evolution-inter-4}, the identity \eqref{bd:correlation-function-for-free-evolution-inter-3} implies \eqref{bd:correlation-function-for-free-evolution-inter-2} as desired.

\medskip

\noindent \underline{The $C_y L^2_t B_2^{-1/2,s_c} $ estimate}. Combining the previous $C_y L^2_t \Dot{H}_x^{-1/2} $ and $C_y L^2_t H_x^{s_c} $ estimates \eqref{bd:correlation-from-free-evolution--1/2} and \eqref{bd:correlation-from-free-evolution-s_c}, it follows that $\mathbb E[\overline{Y(x)}S(t)T_y Z_0(x)]\in C_y L^2_t B_2^{-1/2,s_c}$ with
$$
\| \mathbb E[\overline{Y(x)}S(t)T_y Z_0(x)] \|_{L^\infty_y L^2_t  B_2^{-1/2,s_c}}\lesssim \| Z_0\|_{L^{2d/(d+2)}_x L^2_\omega }+\| Z_0\|_{H^{s_c}_x L^2_\omega }.
$$

\end{proof}

\subsection{Conclusion and proof of the main Theorem}

In this subsection, we finish the proof of Theorem \ref{th:main}. The proof of the theorem relies on solving the fixed point problem (\ref{ptfixe2}). First we write the fixed point argument and then we prove the scattering result. 

We set: 

$$L= \begin{pmatrix}
0 & L_1+L_2\\ 0 & L_3+L_4
\end{pmatrix}.$$

Using Corollary \ref{corollary:l1} and \ref{corollary:l2} we have that $L_1+L_2$ is continuous from $E_V$ to $E_Z$, and using Proposition \ref{prop:invl3-l4}, $1-L_3-L_4$ is continuous, invertible and of continuous inverse on $E_V$. Then $1-L$ is continuous, invertible and of continuous inverse on $E_Z\times E_V$. Thus the problem (\ref{ptfixe2}) can be rewritten as: 

\begin{equation} \label{ptfixe3}
\begin{pmatrix} Z\\ V\end{pmatrix} = (1-L)^{-1} \mathcal B_{Z_0} \begin{pmatrix} Z\\ V\end{pmatrix} = (1-L)^{-1}\begin{pmatrix} \mathcal B^{(1)}_{Z_0}(Z,V)),  \\
\mathcal B^{(2)}_{Z_0}(Z,V)
    \end{pmatrix} 
\end{equation}
where we have set
\begin{align*}
\mathcal B^{(1)}_{Z_0}(Z,V)= S(t)Z_0+Q_1(Z,V)+Q_2(Z,V),\\
\mathcal B^{(2)}_{Z_0}(Z,V) =
\E[\overline{Y(x+y)}S(t)Z_0(x)+\overline{S(t)Z_0(x+y)}Y(x)] \\
+Q_3(Z,V)+Q_4(Z,V)+\E[\overline{Z(x+y)}Z(x)].
\end{align*}

We define the mapping: 

$$\Phi[Z_0]:\left\{ \begin{array}{rcl}
     E_Z\times E_V & \to & E_Z\times E_V\\
     \begin{pmatrix}
     Z \\ V
     \end{pmatrix} & \mapsto & (1-L)^{-1}\mathcal B_{Z_0} \begin{pmatrix} Z\\ V\end{pmatrix}
\end{array}
\right . . $$

Let us denote $E_0:=L_\omega^2H^{s_c}$ the space for the initial datum. We are going to show that for $Z_0$ small enough in $E_0$, the mapping $\Phi[Z_0]$ is a contraction on $B_{E_Z\times E_V}(0,R\lvert\lvert Z_0 \rvert\rvert_{E_0})=:B$ for some constant $R>0$. For simplification we denote the norm $\lvert\lvert \cdot\rvert\rvert_{E_Z\times E_V}$ by $\lvert\lvert\cdot\rvert\rvert$. 

By Corollary \ref{corollary:l1} and \ref{corollary:l2} and Proposition \ref{prop:invl3-l4} we have: 

$$\bigg\lvert\bigg\lvert\Phi[Z_0] \begin{pmatrix}
     Z \\ V
     \end{pmatrix}\bigg\rvert\bigg\rvert \lesssim \bigg\lvert\bigg\lvert B_{Z_0} \begin{pmatrix} Z\\ V\end{pmatrix} \bigg\rvert\bigg\rvert. $$
     
We decompose $\mathcal B_{Z_0}$ as: 
$$\mathcal B_{Z_0}=C_{Z_0}+Q,$$

where: $$C_{Z_0}=\begin{pmatrix}
S(t)Z_0 \\ \E[\overline{Y(x+y)}S(t)Z_0(x)+\overline{S(t)Z_0(x+y)}Y(x)]
\end{pmatrix}$$ is the constant part and: $$Q(Z,V)=\begin{pmatrix}
Q_1(Z,V)+Q_2(Z,V) \\Q_3(Z,V)+Q_4(Z,V)
\end{pmatrix},$$the quadratic part.

By Propositions \ref{prop:ID-Z} and \ref{pr:free-evolution-Z0} we have: 

$$\lvert\lvert C_{Z_0}\rvert\rvert\lesssim \lvert\lvert Z_0\rvert\rvert_{Z_0}. $$

By Propositions \ref{prop:estQ34} and \ref{prop:estQ12} we have that for any $(Z,V)\in B$: 

$$\lvert \lvert Q(Z,V)\rvert\rvert\lesssim \lvert\lvert Z\rvert\rvert_{E_Z}\lvert\lvert V\rvert\rvert_{E_V}\leq R^2 \lvert\lvert Z_0\rvert\rvert_{E_0}^2.$$

Moreover, by bilinearity of $(Z,V)\mapsto Q_k(Z,V)$ for $k=1,2,3,4$ we get that for any $(Z,V), (Z',V')\in B$: 

$$ \lvert\lvert Q(Z,V)-Q(Z',V')\rvert\rvert\lesssim \lvert\lvert (Z,V)-(Z',V')\rvert\rvert (\lvert\lvert (Z,V)\rvert\rvert + \lvert\lvert (Z',V')\rvert\rvert) $$

thus: 

$$\lvert\lvert Q(Z,V)-Q(Z',V')\rvert\rvert \lesssim R\lvert\lvert Z_0\rvert\rvert_{Z_0}\lvert\lvert (Z,V)-(Z',V')\rvert\rvert.$$

From the above estimates, one gets that $\Phi[Z_0]$ is a contraction on $B_{E_Z\times E_V}(0,R\lvert\lvert Z_0\rvert\rvert_{E_0})$, for some universal constant $R>0$, for $\lvert\lvert Z_0\rvert\rvert_{Z_0}$ small enough. By the Banach's fixed point theorem, we get the existence and uniqueness of a solution to (\ref{ptfixe2}) in $B_{E_Z\times E_V}(0,R\lvert\lvert Z_0\rvert\rvert_{E_0})$. 

We can now prove the scattering result of the theorem. We write: 

\begin{align*}
    & Z(t)=S(t)\bigg( Z_0-i\int_0^\infty S(-\tau)\Big[\big(w*V(\cdot,0)\big)Y + \int w(z)V(x,z)Y(x+z)\, dz \Big]\\
       & \qquad \qquad -i\int_0^\infty S(-\tau)\Big[ \big(w*V(\cdot,0)\big)Z + \int w(z)V(x,z)Z(x+z)\, dz\Big] \\ 
       & \qquad \qquad +i\int_t^\infty S(-\tau)\Big[\big(w*V(\cdot,0)\big)Y + \int w(z)V(x,z)Y(x+z)\, dz \Big]\\
       & \qquad \qquad +i\int_t^\infty S(-\tau)\Big[ \big(w*V(\cdot,0)\big)Z + \int w(z)V(x,z)Z(x+z)\, dz\Big]\bigg ).
\end{align*}

By Corollary \ref{corollary:l1} and \ref{corollary:l2} we get that $\int_0^\infty S(-\tau)\Big[\big(w*V(\cdot,0)\big)Y + \int w(z)V(x,z)Y(x+z)\, dz \Big]\in L_\omega^2,H^{s_c}$ and that: 
$$\bigg\lvert\bigg\lvert \int_t^\infty S(-\tau)\Big[\big(w*V(\cdot,0)\big)Y + \int w(z)V(x,z)Y(x+z)\, dz \Big] \bigg\rvert\bigg\rvert_{L_\omega^2,H^{s_c}}\to 0, $$
as $t\to +\infty$.

Moreover, by Proposition \ref{prop:estQ12} we get that $\int_0^\infty S(-\tau)\Big[\big(w*V(\cdot,0)\big)Z + \int w(z)V(x,z)Z(x+z)\, dz \Big]\in L_\omega^2,H^{s_c}$ and that: 
$$\bigg\lvert\bigg\lvert \int_t^\infty S(-\tau)\Big[\big(w*V(\cdot,0)\big)Z + \int w(z)V(x,z)Z(x+z)\, dz \Big] \bigg\rvert\bigg\rvert_{L_\omega^2,H^{s_c}}\to 0, $$
as $t\to +\infty$.

Therefore, there exists $Z_\infty \in \ L_\omega^2,H^{s_c}$ such that, as $t\to \infty$: $$Z(t)=S(t)Z_\infty +o_{L_\omega^2,H^{s_c}}(1). $$

Thus, by definition of $Z$, we get: $$X(t)=Y+S(t)Z_\infty +o_{L_\omega^2,H^{s_c}}(1),$$

concluding the proof of Theorem \ref{th:main}.
\newpage

\bibliographystyle{amsplain}
\bibliography{biblio}

\providecommand{\bysame}{\leavevmode\hbox to3em{\hrulefill}\thinspace}
\providecommand{\MR}{\relax\ifhmode\unskip\space\fi MR }
\providecommand{\MRhref}[2]{%
  \href{http://www.ams.org/mathscinet-getitem?mr=#1}{#2}
}
\providecommand{\href}[2]{#2}
\begin{thebibliography}{10}

\bibitem{ABZ15}
Thomas Alazard, Nicolas Burq, and Claude Zuily, \emph{A stationary phase type
  estimate}, Proceedings of the American Mathematical Society \textbf{145}
  (2017), no.~7, 2871--2880.

\bibitem{BBPPT}
Volker Bach, S{\'e}bastien Breteaux, S{\"o}ren Petrat, Peter Pickl, and Tim
  Tzaneteas, \emph{Kinetic energy estimates for the accuracy of the
  time-dependent hartree--fock approximation with coulomb interaction}, Journal
  de Math{\'e}matiques Pures et Appliqu{\'e}es \textbf{105} (2016), no.~1,
  1--30.

\bibitem{BardosDerivation}
Claude Bardos, Fran{\ifmmode\mbox{\c{c}}\else\c{c}\fi}ois Golse, Alex~D.
  Gottlieb, and Norbert~J. Mauser, \emph{{Mean field dynamics of fermions and
  the time-dependent Hartree{\textendash}Fock equation}}, J. Math. Pures Appl.
  \textbf{82} (2003), no.~6, 665--683.

\bibitem{Benedikter2016Dec}
Niels Benedikter, Vojkan
  Jak{\ifmmode\check{s}\else\v{s}\fi}i{\ifmmode\acute{c}\else\'{c}\fi},
  Marcello Porta, Chiara Saffirio, and Benjamin Schlein, \emph{{Mean-Field
  Evolution of Fermionic Mixed States}}, Commun. Pure Appl. Math. \textbf{69}
  (2016), no.~12, 2250--2303.

\bibitem{Benedikter_2016}
Niels Benedikter, Marcello Porta, Chiara Saffirio, and Benjamin Schlein,
  \emph{From the hartree dynamics to the vlasov equation}, Archive for Rational
  Mechanics and Analysis \textbf{221} (2016), no.~1, 273--334.

\bibitem{benedikter2014mean}
Niels Benedikter, Marcello Porta, and Benjamin Schlein, \emph{Mean--field
  evolution of fermionic systems}, Communications in Mathematical Physics
  \textbf{331} (2014), 1087--1131.

\bibitem{BDF}
Antonio Bove, Giuseppe Da~Prato, and Guido Fano, \emph{An existence proof for
  the hartree-fock time-dependent problem with bounded two-body interaction},
  Communications in mathematical physics \textbf{37} (1974), 183--191.

\bibitem{BDF2}
\bysame, \emph{On the hartree-fock time-dependent problem}, Communications in
  mathematical physics \textbf{49} (1976), 25--33.

\bibitem{cances2014mathematical}
E~Cances, \emph{Mathematical models and numerical methods for electronic
  structure calculation}, Proceedings of the International Congress of
  Mathematicians (Seoul 20140), IV, 2014, pp.~1017--1042.

\bibitem{CDL}
Eric Canc{\`e}s, Am{\'e}lie Deleurence, and Mathieu Lewin, \emph{A new approach
  to the modeling of local defects in crystals: The reduced hartree-fock case},
  Communications in Mathematical Physics \textbf{281} (2008), 129--177.

\bibitem{CF}
Eric Canc{\`e}s and Gero Friesecke, \emph{Density functional theory: Modeling,
  mathematical analysis, computational methods, and applications}, Springer
  Nature, 2023.

\bibitem{chadam1976time}
John~M Chadam, \emph{The time-dependent hartree-fock equations with coulomb
  two-body interaction}, Communications in mathematical physics \textbf{46}
  (1976), 99--104.

\bibitem{CG}
John~M Chadam and Robert~T Glassey, \emph{Global existence of solutions to the
  cauchy problem for time-dependent hartree equations}, Journal of Mathematical
  Physics \textbf{16} (1975), no.~5, 1122--1130.

\bibitem{CHP}
Thomas Chen, Younghun Hong, and Nata{\v{s}}a Pavlovi{\'c}, \emph{Global
  well-posedness of the nls system for infinitely many fermions}, Archive for
  rational mechanics and analysis \textbf{224} (2017), 91--123.

\bibitem{CHP2}
\bysame, \emph{On the scattering problem for infinitely many fermions in
  dimensions {$d\geq 3$} at positive temperature}, Annales de l'Institut Henri
  Poincar{\'e} C, Analyse non lin{\'e}aire, vol.~35, Elsevier, 2018,
  pp.~393--416.

\bibitem{Saffirio}
Jacky~J Chong, Laurent Lafleche, and Chiara Saffirio, \emph{Global-in-time
  semiclassical regularity for the hartree--fock equation}, Journal of
  Mathematical Physics \textbf{63} (2022), no.~8.

\bibitem{CdS}
Charles Collot and A-S de~Suzzoni, \emph{Stability of equilibria for a hartree
  equation for random fields}, Journal de Math{\'e}matiques Pures et
  Appliqu{\'e}es \textbf{137} (2020), 70--100.

\bibitem{CdS2}
Charles Collot and Anne-Sophie de~Suzzoni, \emph{Stability of steady states for
  hartree and schr{\"o}dinger equations for infinitely many particles}, Annales
  Henri Lebesgue \textbf{5} (2022), 429--490.

\bibitem{dS}
Anne-Sophie de~Suzzoni, \emph{Sur les syst{\`e}mes de fermions {\`a} grand
  nombre de particules: un point de vue probabiliste}, S{\'e}minaire Laurent
  Schwartz—EDP et applications (2016), 1--12.

\bibitem{EESY}
Alexander Elgart, L{\'a}szl{\'o} Erd{\H{o}}s, Benjamin Schlein, and Horng-Tzer
  Yau, \emph{Nonlinear hartree equation as the mean field limit of weakly
  coupled fermions}, Journal de math{\'e}matiques pures et appliqu{\'e}es
  \textbf{83} (2004), no.~10, 1241--1273.

\bibitem{FrohlichDerivationFermi}
J{\ifmmode\ddot{u}\else\"{u}\fi}rg Fr{\ifmmode\ddot{o}\else\"{o}\fi}hlich and
  Antti Knowles, \emph{{A Microscopic Derivation of the Time-Dependent
  Hartree-Fock Equation with Coulomb Two-Body Interaction}}, J. Stat. Phys.
  \textbf{145} (2011), no.~1, 23--50.

\bibitem{GIMS98}
Ingenuin Gassner, Reinhard Illner, Peter Markowich, and Christian Schmeiser,
  \emph{Semiclassical, {$t\rightarrow \infty$} asymptotics and dispersive
  effects for hartree-fock systems}, M2AN \textbf{32} (1998), no.~6, 699--713.

\bibitem{GinibreVelo}
J.~Ginibre and G.~Velo, \emph{Scattering theory in the energy space for a class
  of nonlinear {S}chr\"{o}dinger equations}, J. Math. Pures Appl. (9)
  \textbf{64} (1985), no.~4, 363--401. \MR{839728}

\bibitem{H}
Sonae Hadama, \emph{Asymptotic stability of a wide class of steady states for
  the hartree equation for random fields}, arXiv preprint arXiv:2303.02907
  (2023).

\bibitem{NAKASTUD}
\bysame, \emph{Asymptotic stability of a wide class of steady states for the
  hartree equation for random fields}, 2023.

\bibitem{Janson}
Svante Janson, \emph{Gaussian hilbert spaces}, no. 129, Cambridge university
  press, 1997.

\bibitem{J}
Joseph~W Jerome, \emph{Time dependent closed quantum systems: nonlinear
  kohn--sham potential operators and weak solutions}, Journal of Mathematical
  Analysis and Applications \textbf{429} (2015), no.~2, 995--1006.

\bibitem{KT98}
Markus Keel and Terence Tao, \emph{Endpoint strichartz estimates}, American
  Journal of Mathematics \textbf{120} (1998), no.~5, 955--980.

\bibitem{LS2}
Mathieu Lewin and Julien Sabin, \emph{The hartree equation for infinitely many
  particles, ii: Dispersion and scattering in 2d}, Analysis \& PDE \textbf{7}
  (2014), no.~6, 1339--1363.

\bibitem{LS1}
\bysame, \emph{The hartree equation for infinitely many particles i.
  well-posedness theory}, Communications in Mathematical Physics \textbf{334}
  (2015), 117--170.

\bibitem{LS2020}
\bysame, \emph{The hartree and vlasov equations at positive density},
  Communications in Partial Differential Equations \textbf{45} (2020), no.~12,
  1702--1754.

\bibitem{LT93}
Pierre-Louis Lions and Thierry Paul, \emph{Sur les mesures de wigner}, Rev.
  Mat. Iberoam. \textbf{9} (1993), no.~3, 553--618.

\bibitem{NY23}
Toan Nguyen and Chanjin You, \emph{Plasmons for the hartree equations with
  coulomb interaction}, arXiv preprint arXiv:2306.03800 (2023), 1--50.

\bibitem{OL23}
Sewook Oh and Sanghyuk Lee, \emph{Uniform stationary phase estimate with
  limited smoothness}, arXiv preprint arXiv:2012.12572 (2023), 1--11.

\bibitem{Petrat2016Mar}
S{\ifmmode\ddot{o}\else\"{o}\fi}ren Petrat and Peter Pickl, \emph{{A New Method
  and a New Scaling for Deriving Fermionic Mean-Field Dynamics}}, Math. Phys.
  Anal. Geom. \textbf{19} (2016), no.~1, 1--51.

\bibitem{porta2017mean}
Marcello Porta, Simone Rademacher, Chiara Saffirio, and Benjamin Schlein,
  \emph{Mean field evolution of fermions with coulomb interaction}, Journal of
  Statistical Physics \textbf{166} (2017), 1345--1364.

\bibitem{PS}
Fabio Pusateri and Israel~Michael Sigal, \emph{Long-time behaviour of
  time-dependent density functional theory}, Archive for Rational Mechanics and
  Analysis \textbf{241} (2021), no.~1, 447--473.

\bibitem{SCB}
Martin Sprengel, Gabriele Ciaramella, and Alfio Borz{\`\i}, \emph{A theoretical
  investigation of time-dependent kohn--sham equations}, SIAM Journal on
  Mathematical Analysis \textbf{49} (2017), no.~3, 1681--1704.

\bibitem{Tao06}
Terence Tao, \emph{{Nonlinear dispersive equations: local and global
  analysis}}, American Mathematical Society, 2006.

\bibitem{Z}
Sandro Zagatti, \emph{The cauchy problem for hartree-fock time-dependent
  equations}, Annales de l'IHP Physique th{\'e}orique, vol.~56, 1992,
  pp.~357--374.

\end{thebibliography}

\end{document}